\documentclass[12pt]{article}
\title{Weakly o-minimal fields have the exchange property but not generic differentiability}
\author{Will Johnson}

\usepackage{amsmath, amssymb, amsthm}    	
\usepackage{fullpage} 	
\usepackage{amscd}
\usepackage{hyperref}
\usepackage[all]{xy}
\usepackage[T1]{fontenc}
\usepackage{lmodern}
\usepackage{centernot}
\usepackage[shortlabels]{enumitem}

\DeclareMathOperator*{\ind}{\raise0.2ex\hbox{\ooalign{\hidewidth$\vert$\hidewidth\cr\raise-0.9ex\hbox{$\smile$}}}}

\newcommand{\alg}{\mathrm{alg}}
\newcommand{\Th}{\operatorname{Th}}
\newcommand{\sgn}{\operatorname{sgn}}

\newcommand{\characteristic}{\operatorname{char}}

\newcommand{\acl}{\operatorname{acl}}
\newcommand{\dcl}{\operatorname{dcl}}

\newcommand{\im}{\operatorname{im}}

\newcommand{\trdeg}{\operatorname{tr.deg}}

\newtheorem{theorem}{Theorem}[section] 

\newtheorem{lemma}[theorem]{Lemma}
\newtheorem{hypothesis}[theorem]{Hypothesis}

\newtheorem{corollary}[theorem]{Corollary}
\newtheorem{fact}[theorem]{Fact}

\newtheorem{conjecture}[theorem]{Conjecture}

\newtheorem{question}[theorem]{Question}
\newtheorem{proposition}[theorem]{Proposition}
\newtheorem{proposition-eh}[theorem]{Proposition(?)}
\newtheorem*{theorem-star}{Theorem}
\newtheorem*{conjecture-star}{Conjecture}
\newtheorem*{lemma-star}{Lemma}
\newtheorem{claim}[theorem]{Claim}

\newtheorem{probably}[theorem]{Probable Theorem}

\theoremstyle{definition}
\newtheorem{definition}[theorem]{Definition}
\newtheorem{example}[theorem]{Example}

\newtheorem{remark}[theorem]{Remark}

\theoremstyle{remark}

\newtheorem*{acknowledgment}{Acknowledgments}

\newcommand{\Qq}{\mathbb{Q}}

\newcommand{\Rr}{\mathbb{R}}

\newcommand{\Zz}{\mathbb{Z}}

\newcommand{\Nn}{\mathbb{N}}
\newcommand{\Cc}{\mathbb{C}}

\newcommand{\Mm}{\mathbb{M}}

\newcommand{\Ll}{\mathcal{L}}

\newcommand{\pp}{\mathfrak{p}}

\newcommand{\ba}{{\bar{a}}}
\newcommand{\bb}{{\bar{b}}}
\newcommand{\bc}{{\bar{c}}}

\newcommand{\bx}{{\bar{x}}}
\newcommand{\by}{{\bar{y}}}

\newenvironment{claimproof}[1][\proofname]
               {
                 \proof[#1]
                 
               }
               {
                 \endproof
               }

\let\phi\varphi
\begin{document}

\maketitle

\begin{abstract}
  We answer two open questions about weakly o-minimal fields posed by Macpherson, Marker, and Steinhorn: whether weakly o-minimal fields have the exchange property and whether they have generic differentiability.  We construct an ordered field $(K,+,\cdot,\le)$ and a function $f : K \to
  K$ such that the expansion $(K,+,\cdot,\le,f)$ has a weakly o-minimal complete theory but $f$ is nowhere
  differentiable.  In an appendix, we prove that algebraic closure has the exchange property in any weakly o-minimal theory of ordered
  fields.
\end{abstract}

\section{Introduction} \label{sec1}
Recall that an ordered structure $(M,\le,\ldots)$ is \emph{weakly
o-minimal} if every unary definable set $D \subseteq M^1$ is a finite
union of convex sets \cite{dickmann}.  A motivating example of a weakly o-minimal theory is the theory RCVF of real closed fields with non-trivial convex valuation rings, equivalent to the ``real closed rings'' of \cite{dickmann-cherlin}.  More generally, Baisalov and Poizat showed that any expansion of an o-minimal structure $M$ by convex unary sets $X \subseteq M$ is weakly o-minimal \cite{baisalov-poizat}.

Macpherson, Marker, and Steinhorn carefully analyzed weakly o-minimal
ordered fields, proving that they are real closed \cite{weakOmin}.  At
the end of their paper, the authors ask a number of questions,
which have remained open for the past 25 years, including the
following two \cite[\S7.3]{weakOmin}:
\begin{quotation}
  ``\textbf{Problem 3}.  Investigate differentiability and analyticity
  properties of definable functions for weakly o-minimal expansions of
  ordered fields.

  \textbf{Problem 4}.  Is there a weakly o-minimal expansion of an ordered field in which algebraic closure does not have the exchange property?''
\end{quotation}
In Appendix~\ref{app}, we will essentially answer Problem 4, showing that weakly o-minimal theories of ordered fields have the exchange property.  The rest of the paper focuses on Problem 3.

As a starting point for Problem 3, one can ask whether weakly
o-minimal fields have \emph{generic differentiability}:
\begin{question} \label{q1}
  Let $(M,+,\cdot,\le,\ldots)$ be a weakly o-minimal expansion of an
  ordered field.  Let $U \subseteq M^n$ be a definable open set.  Let
  $f : U \to M^m$ be a definable function.  Is there a dense open
  subset $U_0 \subseteq U$ such that $f$ is differentiable on $U_0$?
\end{question}
For \emph{o-minimal} fields, generic differentiability holds
\cite[Chapter~7]{lou-o-minimality}.

In the intervening years, the class of \emph{dp-minimal
theories}~\cite{dpExamples} has been isolated as a common
generalization of strong minimality, o-minimality, weak o-minimality,
and many of their variants such as C-minimality~\cite{c-source} and
P-minimality~\cite{p-min}.  In~\cite{dpm1}, it was shown that any
dp-minimal expansion of a field $(K,+,\cdot,\ldots)$ admits a unique
definable field topology, unless $K$ is strongly minimal.\footnote{For
weakly o-minimal fields, this canonical topology is the order
topology, and for C-minimal and P-minimal fields, it's the valuation
topology.}  Moreover, the boundary of any unary definable set $D
\subseteq K$ is finite~\cite[Theorem~1.3]{dpm1}, a property called
``topological minimality'' by Mathews~\cite{mathews} and
``viscerality'' by Dolich and Goodrick~\cite{visceral}.  Combined with
results of Simon and Walsberg~\cite{simonWalsberg}, it follows that
dp-minimal fields have many ``tame topology'' properties similar to
o-minimal fields.  In particular, definable functions are generically
continuous.  What about generic differentiability?
\begin{question} \label{q2}
  Let $(K,+,\cdot,\ldots)$ be an expansion of a field with
  $\characteristic(K)=0$.  Suppose $K$ is dp-minimal but not strongly
  minimal, so that the canonical topology exists.  Let $U \subseteq
  K^n$ be a definable open set and $f : U \to K^m$ be a definable
  function.  Is there a dense open definable set $U_0 \subseteq U$
  such that $f$ is differentiable on $U_0$?
\end{question}
\begin{remark}
  The assumption $\characteristic(K)=0$ is necessary because the
  function $f(x) = \sqrt[p]{x}$ is nowhere differentiable in any
  dp-minimal field of characteristic $p$.  One such field is the
  C-minimal theory ACVF$_{p,p}$ of algebraically closed valued fields
  of characteristic $p$.
\end{remark}
Question~\ref{q2} has remained open until now, and positive answers
are known in many cases:
\begin{itemize}
\item O-minimal fields~\cite[Chapter~7]{lou-o-minimality}.
\item P-minimal fields \cite{leenknegt,own-P-min}.
\item Definably Cauchy complete C-minimal fields of characteristic 0
  \cite[Theorem~9.5]{own-C-min}.
\item Definably Cauchy complete weakly o-minimal fields, by
  forthcoming work of Acosta, Halevi, Hasson, and Peterzil.
\end{itemize}
Why do Questions~\ref{q1} and \ref{q2} matter?  One reason is that
generic differentiability has applications to definable groups and
fields.  For example, Peterzil, Pillay, and Starchenko used generic
differentiability in o-minimal fields to classify definably simple
groups in o-minimal structures~\cite{deflysimple}, going far beyond
what is known in the analogous situation of strongly minimal
structures.  Similarly, Halevi, Hasson, and Peterzil used generic
differentiability in P-minimal theories to classify interpretable
fields in P-minimal theories (see \cite[Theorem~1(3)]{hhp} and
\cite[Corollary~1.2]{own-P-min}).  The point is that generic
differentiability lets one put a definable $C^1$-manifold structure on
any definable group, enabling tools like the adjoint representation
and Lie theory.

Unfortunately, we will answer Questions~\ref{q1} and \ref{q2}
negatively in this paper:
\begin{theorem}[{= Theorem~\ref{ending}}]
  There is a structure $(M,+,\cdot,\le,f)$ where $(M,+,\cdot,\le)$ is
  a real closed field and $f : M_{>0} \to M_{>0}$ is a unary function,
  such that the following things hold:
  \begin{itemize}
  \item The structure $(K,+,\cdot,\le,f)$ is weakly o-minimal, and its
    complete theory is weakly o-minimal.  In particular,
    $(K,+,\cdot,\le,f)$ is dp-minimal.
  \item The function $f$ is nowhere differentiable.
  \end{itemize}
\end{theorem}
We'll discuss what this means for definable groups and fields in
concluding Section~\ref{conclusion}.
\subsection*{Outline of the construction}
The idea is simple.  Let $\alpha$ be a positive real number which is
exponentially transcendental, i.e., \emph{not} definable in the
structure $(\Rr,+,\cdot,\exp)$.  We will construct a real closed
subfield $M \subseteq \Rr$ with the following properties:
\begin{enumerate}
\item The set $M_{>0}$ of positive elements of $M$ is closed under the
  function $f(x) = x^\alpha$.
\item The expansion $(M,+,\cdot,f)$ is weakly o-minimal.
\item The number $\alpha$ is \emph{not} in $M$.
\end{enumerate}
The derivative $f'(x)$ exists in $\Rr$, and is merely
\begin{equation*}
  f'(x) = \alpha x^{\alpha-1} = \alpha f(x) / x.
\end{equation*}
But note that $f'(a) \notin M$ for any $a \in M_{>0}$, by points (1) and (3).  \emph{Within} the structure $M$, the function $f$ is
nowhere differentiable.

The difficulty is producing a field $M$ with properties (1)--(3).
Already the combination (1) plus (3) is slightly tricky to satisfy.  For
example, if $5^{\alpha}-3^{\alpha}$ happened by coincidence to equal
$\alpha$, then (1) would imply $\alpha \in M$, contradicting (3).  Of
course, this coincidence is ruled out by choosing $\alpha$ to be
exponentially transcendental.

When one adds condition (2) (weak o-minimality), things become
significantly harder.  For example, weak o-minimality forces $M$ to be
closed under not only $f$ but also $f^{-1}(x) = \sqrt[\alpha]{x}$.  Similarly, we need more complicated functions like $(x+1)^\alpha - x^\alpha$ to be invertible at most points.
In general, it seems difficult to satisfy (1) and (2) unless $M$ is closed
under any function ``implicitly defined'' using $x^\alpha$.  For example, if the system of equations
\begin{gather*}
  x^\alpha + y^\alpha = 5 \\
  x - y = 1
\end{gather*}
has a regular\footnote{Here, a ``regular solution'' means that the matrix of partial derivatives $\begin{pmatrix} \alpha x^{\alpha-1} & \alpha y^{\alpha-1} \\ 1 & -1
\end{pmatrix}$ is invertible at $(x,y)=(a,b)$.} solution $(x,y) = (a,b)$ then we probably want $a$ and $b$
to be in the field $M$.

Initially, I tried to take $M$ to be the field containing all
coordinates of regular solutions of system of equations of the form
\begin{gather*}
  P_1(x_1,\ldots,x_n) = 0 \\
  P_2(x_1,\ldots,x_n) = 0 \\
  \vdots \\
  P_n(x_1,\ldots,x_n) = 0,
\end{gather*}
where each $P_1,\ldots,P_n$ is a $\Qq$-polynomial in
$x_1,\ldots,x_n,x_1^\alpha,\ldots,x_n^\alpha$.  Using work of Bays, Kirby, and Wilkie~\cite{bkw}, one can prove that this
field doesn't contain $\alpha$ (see Proposition~\ref{lockdown}).
Unfortunately, I ran into too many obstacles when trying to prove weak
o-minimality of this field, and so we will use a slightly different
strategy.

Let $K$ be the prime model of the o-minimal structure
$(\Rr,+,\cdot,\le,x^\alpha)$, i.e., $\dcl(\varnothing)$ in this
structure.  By combining results of Miller, Jones, and Wilkie
\cite{miller,jones-wilkie}, one can show (Proposition~\ref{jones-wilkie-repeat}) that the positive elements of
$K$ are precisely the elements which occur as coordinates of regular
solutions of systems of equations
\begin{gather*}
  P_1(x_1,\ldots,x_n) = 0 \\
  P_2(x_1,\ldots,x_n) = 0 \\
  \vdots \\
  P_n(x_1,\ldots,x_n) = 0,
\end{gather*}
where the $P_i$ are polynomials in
$x_1,\ldots,x_n,x_1^\alpha,\ldots,x_n^\alpha$ with coefficients in
\underline{$\Qq[\alpha]$} rather than $\Qq$ as before.  For example,
$\alpha$ itself belongs to $K$, since it's a regular solution of the
equation $x-\alpha=0$.  We need a way to separate the ``bad'' elements
of $K$, like $\alpha$, from the ``good'' elements, like those defined
by systems of equations with $\Qq$-coefficients.

Given a system of equations like
\begin{align}
  x^\alpha - y^\alpha &= 0 \label{eq1} \\
  x+y &= 8 \label{eq2} \\
  z - \alpha x &= 0, \label{eq3}
\end{align}
and a regular solution $(x,y,z)=(a,b,c)$, the trick is to replace $\alpha$ with a moving parameter $t$ whenever it occurs as a coefficient (rather than an exponent)
\begin{align*}
  x^\alpha - y^\alpha &= 0 \\
  x+y &= 8 \\
  z - tx &= 0
\end{align*}
and then look at how the solution $(a(t),b(t),c(t))$ varies as a
function of $t$.  If the derivative $a'(t)$ vanishes at $t=\alpha$,
then $a$ is ``good''.  Otherwise, $a$ is ``bad''.\footnote{In the
example just given, $a(t) = a$ and $b(t) = b$ and $c(t) = at$, so
$a'(t)=0$, $b'(t)=0$, and $c'(t)=a$.  Thus $a,b$ are ``good'' and $c$
is ``bad''.}  The surprising thing is that the map sending $a$ to its
``derivative'' $a'(t)$ is well-defined, giving the following key result:
\begin{theorem}[{= Theorem~\ref{the-key}}] \label{the-key-0}
  There is a unique derivation $\delta : K \to K$ satisfying the
  constraints
  \begin{align*}
    \delta \alpha &= 1 \tag{$\ast$} \\
    \delta x^\alpha &= \alpha x^{\alpha - 1} \delta x \text{ for } x > 0. \tag{$\dag$}
  \end{align*}
\end{theorem}
For instance, the reader may wish to check that---assuming
Theorem~\ref{the-key-0}---any regular solution $(a,b,c)$ of
Equations~\ref{eq1}--\ref{eq3} satisfies $\delta a = 0$, $\delta b =
0$, and $\delta c = a$.  Our weakly o-minimal field $M$ will be the
``good'' elements with $\delta x = 0$, rather than the ``bad''
elements with $\delta x \ne 0$.
\begin{remark}
  There is something off about the conditions on the derivation
  $\delta$ in Theorem~\ref{the-key-0}.  Normally, the derivative of
  $x^y$ is $yx^{y-1}x' + (\log x)x^y y'$, which only equals
  $yx^{y-1}x'$ if $y'=0$ or $x=1$.  So the identity $\delta x^\alpha =
  \alpha x^{\alpha-1} \delta x$ suggests that $\delta \alpha = 0$, when actually $\delta \alpha = 1$ instead.

  The point is that $\delta$ treats $\alpha$ as a constant when it
  appears as an exponent (in expressions like $x^\alpha$), but as a
  moving parameter when $\alpha$ appears as a coefficient (in
  expressions like $\alpha x$).  Since our goal is to separate the map
  $x^\alpha$ from the number $\alpha$, this is exactly the behavior we
  want.
\end{remark}
We will call $\delta$ the \emph{twisted derivation} because of this
strange behavior.  We will prove the following definability theorem
for $\delta$:
\begin{theorem}[{= Theorem~\ref{definability}}] \label{definability-0}
  Let $f : K^n \to K$ be definable in the structure
  $(K,+,\cdot,x^\alpha)$.  Then there is another definable function $g
  : K^{2n} \to K$ such that
  \begin{equation*}
    \delta f(a_1,\ldots,a_n) = g(a_1,\ldots,a_n,\delta a_1, \ldots, \delta a_n) \text{ for any } a_1,\ldots,a_n \in K.
  \end{equation*}
\end{theorem}
\begin{example}
  Suppose $f(x,y)$ is the function $\alpha x + y^\alpha$.  Using the
  defining properties of $\delta$ in Theorem~\ref{the-key-0}, one sees that
  \begin{equation*}
    \delta f(x,y) = \alpha \delta x + x + \alpha y^{\alpha-1} \delta y,
  \end{equation*}
  so if we take $g(x,y,x',y') = \alpha x' + x + \alpha y^{\alpha-1} y'$, then
  \begin{equation*}
    \delta f(x,y) = g(x,y,\delta x, \delta y).
  \end{equation*}
\end{example}
Assuming Theorems~\ref{the-key-0} and \ref{definability-0}, the weak
o-minimality of $(M,+,\cdot,x^\alpha)$ follows by a very easy proof
given in Section~\ref{easy-part}, which you could go read right now if
you wanted.  All the difficulty goes into proving
Theorems~\ref{the-key-0} and \ref{definability-0}, and so we now
sketch their proofs.

First, we need a recipe for calculating $\delta$.  This was discused
earlier: the results of Miller, Jones, and Wilkie show that every
positive element of $K$ can be implicitly defined by a system of
equations.  One formally applies $\delta$ to these equations, and
solves for $\delta x_1, \delta x_2, \ldots, \delta x_n$.
\begin{example}
  Suppose $x,y \in K$ satisfy the equations
  \begin{align*}
    \alpha x^\alpha + y &= 3 \\
    x - y^\alpha &= 1.
  \end{align*}
  Formally applying $\delta$, we get
  \begin{align*}
    x^\alpha + \alpha^2 x^{\alpha - 1} \delta x + \delta y = 0 \\
    \delta x - \alpha y^{\alpha-1} \delta y = 0,
  \end{align*}
  and so
  \begin{equation*}
    \begin{pmatrix}
      \delta x \\ \delta y
    \end{pmatrix}
    = 
    \begin{pmatrix}
      \alpha^2 x^{\alpha-1} & 1 \\ 1 & -\alpha y^{\alpha-1}
    \end{pmatrix}^{-1} 
    \begin{pmatrix}
      -x^\alpha \\ 0
    \end{pmatrix}.
  \end{equation*}
  The matrix which needs to be invertible is precisely the matrix of
  partial derivatives of the original equation, which is invertible so
  long as $(x,y)$ was a \emph{regular} solution of the system of
  equations.
\end{example}
\begin{example} \label{exdos}
  Suppose $x \in K$ satisfies the equation
  \begin{equation*}
    x^\alpha + \alpha x + x = 3.
  \end{equation*}
  Formally applying $\delta$, we get
  \begin{equation*}
    \alpha x^{\alpha-1} \delta x + x  + \alpha \delta x + \delta x = 0,
  \end{equation*}
  and so
  \begin{equation*}
    \delta x = \frac{-x}{\alpha x^{\alpha-1} + \alpha + 1}.
  \end{equation*}
  Again, the denominator is the derivative $f'(x)$ for $f(x) =
  x^\alpha + \alpha x + x$.
\end{example}
As mentioned earlier, the funny thing in these calculations is that
one treats $\alpha$ as a variable with derivative 1 when it occurs as
a coefficient, but as a constant when it appears as an exponent.  For
instance, in Example~\ref{exdos} we are really looking at the solution
of $x^\alpha + tx + x = 3$ as $t$ ranges over a neighborhood of
$\alpha$ in $\Rr$, and taking the derivative $dx/dt$.

This recipe shows that $\delta$ is \emph{unique}, if it exists.  But \emph{existence} of $\delta$ is more subtle.  We need to show that the recipe does not give inconsistent values for $\delta x$.

For example, we would run into trouble if there were some $x \in K$
such that
\begin{gather}
  x^\alpha + \alpha x + x = 3 \label{eqn1} \\
  3x + \alpha = 4 \label{eqn2}
\end{gather}
but
\begin{equation*}
  \frac{-x}{\alpha x^{\alpha-1} + \alpha + 1} \ne -1/3.
\end{equation*}
The left side is the value of $\delta$ calculated using
Equation~\ref{eqn1}, and the right side is calculated using
Equation~\ref{eqn2}; if they disagree then no possible choice of
$\delta$ can satisfy the requirements of Theorem~\ref{the-key-0}.

Of course, in reality there is no problem because the genericity of
$\alpha$ ensures that the solutions of Equation~\ref{eqn1} do not
intersect the solutions of Equation~\ref{eqn2}.  Unfortunately, we
need something like Ax-Schanuel \cite{ax-schanuel} to make this
argument work precisely in general.  More precisely, we need the
Bays-Kirby-Wilkie version of Schanuel's conjecture for $x^\alpha$
\cite{bkw}.

As an example of how (Ax-)Schanuel intervenes, suppose we had started
from $(\Rr,+,\cdot,\alpha^x)$ rather than $(\Rr,+,\cdot,x^\alpha)$.\footnote{If we could use $(\Rr,+,\cdot,\alpha^x)$ instead of $(\Rr,+,\cdot,x^\alpha)$, our life would be much easier in some ways, because $\alpha^x$ is defined on all of $\Rr$ rather than only on positive numbers.  In particular, all of \S\ref{unneed} could be skipped.}
Let $K'$ be the prime model of $(\Rr,+,\cdot,\alpha^x)$.  Note $\log
\alpha \in K'$ because the derivative of $\alpha^x$ is $(\log \alpha)
\alpha^x$.  For various reasons, the appropriate analog of
Theorem~\ref{the-key-0} is probably the following:
\begin{hypothesis} \label{hyp2}
  There is a unique derivation $\delta: K' \to K'$ satisfying the
  following conditions:
  \begin{gather*}
    \delta (\alpha^x) = (\log \alpha) \cdot \alpha^x \delta x \\
    \delta (\log \alpha) = 1.
  \end{gather*}
\end{hypothesis}
If Hypothesis~\ref{hyp2} is true, then one can use the fact that $e^x
= \alpha^{\frac{x}{\log \alpha}}$ to calculate \[\delta(e^e) =
\frac{-2e^{e+1}}{\log \alpha} \ne 0.\] This would imply that $e^e$ is
transcendental, answering a major open problem in transcendence
theory.  I suspect that Hypothesis~\ref{hyp2} is difficult to prove
without Schanuel's conjecture.

\subsection*{Outline of the paper}
In Section~\ref{sec2}, we review prior results about the structure
$(\Rr,+,\cdot,x^\alpha)$, and in Section~\ref{unneed} we use the
technique of Jones and Wilkie \cite{jones-wilkie} to characterize
definable closure in this structure.  In Section~\ref{schan} we tease
out the consequence of the Bays-Kirby-Wilkie~\cite{bkw} Schanuel
property for $x^\alpha$ that will be needed to show that the twisted
derivation $\delta$ is well-defined (i.e., the recipe for calculating
$\delta$ is self-consistent).  Section~\ref{delta-core} is the
technical heart of the paper, where we construct the twisted
derivative $\delta$.  The remaining easy parts of the proof are
completed in Section~\ref{easy-part}.  Section~\ref{conclusion} gives
some directions for further research, and explains the research
questions which led to the discovery of this example.

Appendix~\ref{app} proves the exchange property in weakly o-minimal
theories of fields.  This result is independent of the rest of the
paper, and is included because of its thematic connection to weak
o-minimality.

\subsection*{Notation and conventions}
``Definable'' means ``definable with parameters'', and ``0-definable''
means ``definable without parameters''.  If we say that a structure
$M$ is a ``field'', we mean that it's an expansion of a field.  If we
want $M$ to have no symbols beyond the language of rings, we say that
$M$ is a ``pure field''.  If $K$ is an ordered field, then $K_{>0}$
denotes the set of positive numbers.

If $(a_1,\ldots,a_n)$ is a row vector, then $(a_1,\ldots,a_n)^{\mathrm{T}}$ denotes its transpose, the column vector \[ \begin{pmatrix} a_1 \\ a_2 \\ \vdots \\ a_n \end{pmatrix}.\]  Notation like $(a_1,\ldots,a_n) \cdot (b_1,\ldots,b_n)^{\mathrm{T}}$ means the inner product $\sum_{i=1}^n a_ib_i$.

If $f_1,\ldots,f_m : \Rr^n \to \Rr$ are functions, then
$J(f_1,\ldots,f_m)$ is the Jacobian matrix, the $m \times n$ matrix
whose $i$th row is the vector $(\partial f_i / \partial x_1, \partial
f_i / \partial x_2, \ldots, \partial f_i / \partial x_n)$.  The
vanishing set of $f_1,\ldots,f_m$ is
\begin{equation*}
  V(f_1,\ldots,f_m) = \{a \in \Rr^n : f_1(a) = f_2(a) = \cdots =
  f_m(a) = 0\}.
\end{equation*}
If $m \le n$, the set of \emph{regular solutions} is
\begin{equation*}
  V^{reg}(f_1,\ldots,f_m) = \{a \in V(f_1,\ldots,f_m) :
  J(f_1,\ldots,f_m) \text{ has rank } m\}.
\end{equation*}
This set is an $(n-m)$-dimensional manifold by the implicit function theorem.

\textbf{For the remainder of the paper}, fix a positive real number $\alpha$
which is exponentially transcendental, i.e., not 0-definable in the
structure $\Rr_{\exp} = (\Rr,+,\cdot,\exp)$.

Let $R_n$ be the ring of functions $(\Rr_{>0})^n \to \Rr$ which are
given as Laurent polynomials in
\begin{equation*}
	x_1, \ldots, x_n, x_1^\alpha, \ldots, x_n^\alpha
\end{equation*}
with coefficients in $\Qq$.  Let $R_n[\alpha]$ be defined similarly,
but allowing $\Qq[\alpha]$-coefficients.  In other words,
\begin{align*}
	R_n &= \Qq[x_1^{\pm 1}, \ldots, x_n^{\pm 1}, x_1^{\pm \alpha}, \ldots, x_n^{\pm \alpha}] \\
	R_n[\alpha] &= \Qq[x_1^{\pm 1}, \ldots, x_n^{\pm 1}, x_1^{\pm \alpha}, \ldots, x_n^{\pm \alpha},\alpha]
\end{align*}
\begin{remark}
	\begin{enumerate}
		\item $R_n[\alpha]$ is closed under partial
		derivatives, but $R_n$ is not.
		\item If $f \in R_n[\alpha]$, then $f$ can be written as
		\begin{equation*}
			f(x_1,\ldots,x_n) = g(\alpha,x_1,\ldots,x_n)
		\end{equation*}
		for some $g \in R_{n+1}$.\footnote{As an example, the function
			$f(x,y) = \alpha x + y^\alpha$ can be written as $g(\alpha,x,y)$
			where $g(w,x,y) = wx + y^\alpha$.}  The right hand side makes sense
		because we chose $\alpha > 0$.
	\end{enumerate}
\end{remark}

\section{Review of prior results} \label{sec2}
Bays, Kirby, and
Wilkie prove the following Schanuel property \cite{bkw}:
\begin{fact} \label{f1}
  If $x_1,\ldots,x_n$ are positive real numbers which are
  multiplicatively independent, then
  $\trdeg(x_1,\ldots,x_n,x_1^\alpha,\ldots,x_n^\alpha/\alpha) \ge n$.
\end{fact}
Consider the structure $\Rr_\alpha := (\Rr,+,\cdot,x^\alpha,\alpha)$, where
the unary function $x^\alpha$ is extended to take the value 0 for $x
\le 0$.
By Wilkie's work on $\Rr_{\exp}$ \cite{wilkie-thm}, one knows that
$\Rr_\alpha$ is o-minimal.  Building on this, Miller proves the
following~\cite[Theorems~3.4, 4.6]{miller}:
\begin{fact} \label{f2}
  The structure $\Rr_\alpha$ is polynomially bounded and model
  complete.
\end{fact}
Although we will not use it, it is worth noting the following theorem
of Jones and Servi \cite{servi}:
\begin{fact} \label{f3}
  The theory of $\Rr_\alpha$ is decidable relative to an oracle for
  the set $\{q \in \Qq : q < \alpha\}$.
\end{fact}
\begin{remark}
  Facts~\ref{f1} and \ref{f3} depend on $\alpha$ being exponentially
  transcendental, but Fact~\ref{f2} does not.
\end{remark}

\section{Definable closure in $\Rr_\alpha$} \label{unneed}
The goal of this section is Proposition~\ref{jones-wilkie-repeat},
which characterizes definable closure in elementary extensions of
$\Rr_\alpha$.  Everything in this section is probably obvious to
experts, and follows easily using the methods of Jones and Wilkie
\cite{jones-wilkie}.  For completeness, we include the proofs.
\begin{remark}
  At first glance, it seems like we should be able to directly cite
  \cite[Theorem~3.2, Corollary~3.4, Corollary~3.5]{jones-wilkie}.
  Unfortunately, $\Rr_\alpha$ is not literally a ``locally polynomial
  bounded structure'' in the sense of \cite[Section~2]{jones-wilkie},
  for the (seemingly frivolous) reason that the smooth function
  $x^\alpha$ is defined on $\Rr_{>0}$ rather than $\Rr$.  Of course,
  one can reduce to \cite[Theorem~3.2]{jones-wilkie} by moving to the
  locally polynomially bounded structure
  $(\Rr,+,\cdot,\alpha,(1+x^2)^\alpha)$, which is definitionally
  equivalent to $\Rr_\alpha$.  In this way, one can reduce
  Proposition~\ref{jones-wilkie-repeat} to
  \cite[Theorem~3.2]{jones-wilkie}.  The reduction is straightforward,
  but surprisingly lengthy, and it turns out to be faster to redo the
  proofs from scratch.
\end{remark}
  
\begin{lemma} \label{nearest}
  Let $M$ be an o-minimal structure expanding a real closed field.
  Let $D \subseteq M^n$ be a closed definable set.  Then there is
  $c \in \Qq^n$ such that there is a unique point in $D$
  maximally close to $c$.
\end{lemma}
Note that if $c \in D$, then $c$ is the point in $D$
maximally close to $c$.
\begin{proof}
  Say that $c \in M^n$ is \emph{good} if there is a unique point in
  $D$ maximally close to $c$, and \emph{bad} otherwise.  The set of
  good points is definable.  If every point in $\Qq^n$ is bad, then
  the set of bad points has non-empty interior\footnote{This can be
  seen easily from cell decomposition.  A non-open cell has dimension
  less than $n$ and covers at most $k^{n-1}$ points in the set
  $\{1,\ldots,k\}^n$.  If $D$ has empty interior, then taking $k$
  larger than the number of cells, $D$ cannot cover
  $\{1,\ldots,k\}^n$.}.  Take $c$ in the interior of the bad points,
  and take $e \in D$ one of the points maximally close to $c$.  Note $c \notin D$ since $c$ is bad.  Let $r
  = ||e - c||$, and let $B$ be the closed ball of radius $r$ around
  $c$.  The interior of $B$ doesn't intersect $D$.  If $\epsilon > 0$
  is small enough, then the point $c' = c + (e - c)\epsilon$ is still
  bad, because $c$ is in the interior of the bad points.  Let $B'$ be
  the closed ball of radius $(1-\epsilon)r$ around $c'$.  Then $e \in
  B'$, and every other point of $B'$ is in the interior of $B$.  Consequently
  $D \cap B' = \{e\}$, and $e$ is the unique point maximally close to
  $c'$, contradicting badness of $c'$.
\end{proof}
\begin{lemma} \label{wilkie}
  Let $M$ be a polynomially bounded o-minimal structure.  Let $a \in
  M^n$ be a point.  Let $R$ be a collection of germs of smooth (i.e.,
  $C^\infty$) definable functions defined on open neighborhoods of
  $a$.  Suppose $R$ is a ring, $R$ is closed under taking partial
  derivatives, and $R$ contains all polynomials over $\Qq$.  Then
  there are $f_1,\ldots,f_m \in R$ such that
  \begin{enumerate}
  \item $a \in V^{reg}(f_1,\ldots,f_m)$.
  \item If $g \in R$ vanishes at $a$, then $g$ vanishes on a
    neighborhood of $a$ in $V^{reg}(f_1,\ldots,f_m)$.
  \item If $g \in R$ and the restriction $g \restriction
    V^{reg}(f_1,\ldots,f_m)$ has a critical point at $a$, then $g
    \restriction V^{reg}(f_1,\ldots,f_m)$ is constant on a
    neighborhood of $a$.
  \end{enumerate}
\end{lemma}
\begin{proof}
  Let $\pp \lhd R$ be the kernel of the evaluation map
  \begin{align*}
    R &\to M \\
    f &\mapsto f(a).
  \end{align*}
  Then $\pp$ is a prime ideal in $R$.  The localization $R_\pp$ is
  again a ring of germs of definable functions defined on
  neighborhoods of $\ba$, closed under partial derivatives.

  Take $f_1,\ldots,f_m \in R_\pp$ such that $a \in
  V^{reg}(f_1,\ldots,f_m)$, with $m$ maximal (possibly $m=0$).  There
  is some $s \in R \setminus \pp$ such that we can write every $f_i$
  as $f_i = g_i/s$ with $g_i \in R$.  Then the $g_i$ vanish at $a$,
  and the matrix of partial derivatives $\frac{\partial g_i}{\partial
    x_j}$ has rank $m$.\footnote{Note that $\frac{\partial
    g_i}{\partial x_j} = s \frac{\partial f_i}{\partial x_j}$ at the
  point $a$, because $f_i(a) = 0$.  So the Jacobian matrix of the
  $g_i$'s is $s(a)$ times the Jacobian matrix of the $f_i$'s, and the
  two matrices have the same rank.}  Replacing the $f_i$ with the
  $g_i$, we may assume the $f_i$ come from $R$.

  Around $a$, the set $V^{reg}(f_1,\ldots,f_m)$ looks locally like an
  $(n-m)$-dimensional manifold, by the implicit function theorem.
  Permuting coordinates, we may assume that the last $m$ columns of
  the Jacobian matrix $J(f_1,\ldots,f_m)$ are linearly independent,
  when evaluated at $a$.  Then $V^{reg}(f_1,\ldots,f_m)$ looks locally
  like the graph of some smooth definable function $\phi : U \to M^m$
  with $U \subseteq M^{n-m}$ an open neighborhood of
  $(a_1,\ldots,a_{n-m})$.  In particular, $\phi(a_1,\ldots,a_{n-m}) =
  (a_{n-m+1},\ldots,a_n)$.

  Let $R'$ be the set of germs of functions around
  $(a_1,\ldots,a_{n-m})$ of the form $g(x,\phi(x))$ with $g \in
  R_\pp$.  Using the implicit definition $\bar{f}(x,\phi(x)) = 0$, one
  can see that the partial derivatives $\frac{\partial
    \phi_j}{\partial x_i}$ are elements of $R'$.\footnote{To calculate
  the derivative of $\phi$ one has to invert a certain matrix---the
  final $m$ columns of $J(f_1,\ldots,f_m)$---but this causes no
  problems because $R_\pp$ is a local ring and the determinant of
  $J(f_1,\ldots,f_m)$ is invertible in $R_\pp$.}  It then follows that
  $R'$ is closed under partial derivatives,
  since \[\frac{\partial}{\partial x_i} g(\bx,\phi(\bx)) =
  \frac{\partial g}{\partial x_i} g(\bx,\phi(\bx)) + \sum_{j=1}^m
  \frac{\partial g}{\partial x_{j+(n-m)}}(\bx,\phi(\bx))
  \frac{\partial \phi_j}{\partial x_i}.\]
  \begin{claim}
    If $g \in R'$ vanishes at $(a_1,\ldots,a_{n-m})$, then $g$ is
    identically zero (as a germ).
  \end{claim}
  \begin{claimproof}
    Otherwise, a theorem of Miller~\cite{miller2} shows that some
    iterated partial derivative of $g$ doesn't vanish at $a$.  (This
    uses the polynomial boundedness.)  So then there is an iterated
    partial derivative $h$ of $g$ with
    \begin{equation*}
      h(a_1,\ldots,a_{n-m}) = 0 \ne \frac{\partial h}{\partial x_i}(a_1,\ldots,a_{n-m}).
    \end{equation*}
    But then $h = \tilde{h}(x,\phi(x))$ for some $\tilde{h} \in R_\pp$,
    and $(a_1,\ldots,a_n) \in V^{reg}(f_1,\ldots,f_m,\tilde{h})$,
    contradicting the maximality of $m$.
  \end{claimproof}
  Part (2) of the Lemma now follows: if $g \in R$ vanishes at
  $(a_1,\ldots,a_n)$, then its image $g(x,\phi(x))$ in $R'$ vanishes
  at $(a_1,\ldots,a_{n-m})$ so $g(x,\phi(x))$ vanishes for all $x$ on
  a neighborhood of $(a_1,\ldots,a_{n-m})$, meaning that $g$ vanishes
  on a neighborhood of $(a_1,\ldots,a_n)$ in
  $V^{reg}(f_1,\ldots,f_m)$.

  For part (3), if $g \in R$ has a critical point on
  $V^{reg}(f_1,\ldots,f_n)$ at $(a_1,\ldots,a_n)$, then the function
  $g(x,\phi(x))$ has a critical point at $(a_1,\ldots,a_{n-1})$.  All
  the partial derivatives vanish.  By the claim, all the partial
  derivatives are identically zero on a neighborhood of
  $(a_1,\ldots,a_{n-1})$.  This implies that $g(x,\phi(x))$ is
  constant on a neighborhood of this point, and so $g$ is constant on
  $V^{reg}(f_1,\ldots,f_n)$ around $(a_1,\ldots,a_n)$.
\end{proof}
Recall the rings $R_n$ and $R_n[\alpha]$ defined at the end of Section~\ref{sec1}:
\begin{align*}
	R_n &= \Qq[x_1^{\pm 1}, \ldots, x_n^{\pm 1}, x_1^{\pm \alpha}, \ldots, x_n^{\pm \alpha}] \\
	R_n[\alpha] &= \Qq[x_1^{\pm 1}, \ldots, x_n^{\pm 1}, x_1^{\pm \alpha}, \ldots, x_n^{\pm \alpha},\alpha]
\end{align*}
Recall that $R_n[\alpha]$ is closed under partial derivatives.
\begin{proposition} \label{jones-wilkie-repeat}
  Let $M$ be an elementary extension of $\Rr_\alpha$.  Suppose
  $a_1,\ldots,a_n,b$ are positive elements with $b \in
  \dcl(a_1,\ldots,a_n)$.  Then there are positive elements
  $b_2,\ldots,b_m$ and functions $f_1,\ldots,f_m \in R_{n+m}[\alpha]$ such
  that
  \begin{equation*}
    (b,b_2,\ldots,b_m) \in V^{reg}(f_1(\ba,-),\ldots,f_m(\ba,-)).
  \end{equation*}
\end{proposition}
\begin{proof}
	Let $B_m$ be the ring of functions $(M_{>0})^m \to M$ of the form $f(\ba,x_1,\ldots,x_m)$, with $f \in R_{n+m}[\alpha]$.  Note $B_m$ is closed under partial derivatives.
  The set $\{b\}$ is $\ba$-definable, so  by model completeness $\{b\}$ is the coordinate projection
  of $V(f) \subseteq M^m$ onto the first coordinate, for some $f \in B_m$.
  Let $\theta$ be the semialgebraic diffeomorphism
  \begin{align*}
    \theta : \Rr_{>0} &\to \Rr \\
    \theta(x) &= x - x^{-1}
  \end{align*}
  Let $D = \{(\theta(x_1),\ldots,\theta(x_m)) : \bx \in V(f)\}$.  By
  Lemma~\ref{nearest}, there is some $c \in \Qq^m$ such that there is
  a unique point $(u_1,\ldots,u_m) \in D$ minimizing
  $||c-(u_1,\ldots,u_m)||^2$.  Equivalently, there is a unique point
  $(b_1,\ldots,b_m) \in V(f)$ minimizing $||c -
  (\theta(b_1),\ldots,\theta(b_m))||^2$.  Since $(b_1,\ldots,b_m) \in
  V(f)$, we must have $b_1 = b$.

  Applying Lemma~\ref{wilkie} to the point $\bb$ and the ring
  $B_m$, we get $f_1,\ldots,f_\ell \in B_m$ with the
  following properties:
  \begin{enumerate}
  \item $\bb \in V^{reg}(f_1,\ldots,f_\ell)$.
  \item If $g \in B_m$ vanishes at $\bb$, then it vanishes on
    a neighborhood of $\bb$ in $V^{reg}(f_1,\ldots,f_\ell)$.
  \item If $g \in B_m$ has a critical point at $\bb$ on the
    manifold $V^{reg}(f_1,\ldots,f_\ell)$, then it's locally
    constant.
  \end{enumerate}
  It remains to show that $\ell = m$.  Applying the second point to $g
  = f$, we see that $f$ vanishes on a neighborhood of $\bb$ in
  $V^{reg}(f_1,\ldots,f_\ell)$.  So $V^{reg}(f_1,\ldots,f_\ell)
  \subseteq V(f)$, at least locally near $\bb$.

  The function $g(x_1,\ldots,x_m) = ||c -
  (\theta(x_1),\ldots,\theta(x_m))||^2$ is in $B_m$, and $\bb$
  is a local minimum of $g$ on $V(f)$, hence on
  $V^{reg}(f_1,\ldots,f_\ell)$.  It follows that $g$ has a critical
  point at $\bb$.  By the third point of Lemma~\ref{wilkie}, $g$ is
  locally constant on a neighborhood of $\bb$ in
  $V^{reg}(f_1,\ldots,f_\ell)$.  If $\ell < m$, then
  $V^{reg}(f_1,\ldots,f_\ell)$ is positive dimensional, so we get
  many points $\bc$ close to $\bb$ which lie on
  $V^{reg}(f_1,\ldots,f_\ell) \subseteq V(f)$ and have $g(\bc) =
  g(\bb)$, contradicting the fact that $\bb$ is the unique minimum
  point for $g$ on $V(f)$.  Instead, we must have $\ell = m$.  Then
  $\bb \in V^{reg}(f_1,\ldots,f_m)$ as desired.
\end{proof}
\begin{corollary} \label{cor1}
  If $b \in \Rr_{>0}$ is 0-definable in $\Rr_\alpha$, then there are
  $b_2,\ldots,b_m \in \Rr_{>0}$ and $g_1,\ldots,g_m \in R_m[\alpha]$
  such that $(b,b_2,\ldots,b_m) \in V^{reg}(g_1,\ldots,g_m)$.
\end{corollary}
\begin{definition} \label{implicit}
  In the structure $\Rr_\alpha$, a 0-definable function $f : U \to
  \Rr_{>0}$ with $U \subseteq (\Rr_{>0})^n$ is \emph{implicitly
  defined} if there are 0-definable functions $f_2,\ldots,f_m :
  U \to \Rr_{>0}$ and functions $g_1,\ldots,g_m \in
  R_{n+m}[\alpha]$ such that
  \begin{equation*}
    (f(\ba),f_2(\ba),\ldots,f_m(\ba)) \in
    V^{reg}(g_1(\ba,-),\ldots,g_m(\ba,-))
  \end{equation*}
  for every $\ba \in U$.
\end{definition}
\begin{corollary} \label{cor2}
  If $U \subseteq (\Rr_{>0})^n$ is open and $f : U \to \Rr_{>0}$ is
  0-definable in $\Rr_\alpha$, then there are 0-definable
  $U_1,\ldots,U_m \subseteq U$ such that $f \restriction U_i$ is
  implicitly defined, and $\dim(U \setminus \bigcup_{i=1}^m U_i) < n$.
\end{corollary}
\begin{proof}
  This follows from Proposition~\ref{jones-wilkie-repeat} by a
  compactness argument, exactly like \cite[Corollaries
    3.4--3.5]{jones-wilkie}.
\end{proof}
Going forward, the only things we will need from this section are Corollaries~\ref{cor1} and \ref{cor2}.

\section{Application of Schanuel} \label{schan}
\begin{remark} \label{sard-idk}
  Let $\ba \in \Rr^n$ be a point.  Let $f : U \to \Rr^k$ be
  semialgebraic over $\Qq$, i.e., $\Qq$-definable in $(\Rr,+,\cdot)$.
  Suppose that $f$ is $C^\infty$ and $\ba \in U$.  Suppose that
  $f(\ba) = \bar{0} \in \Rr^k$, and the Jacobian derivative matrix of
  $f$ at $\ba$ has rank $k$ (i.e., the map on tangent spaces is onto).
  Then $\trdeg(\ba/\Qq) \le n-k$.
\end{remark}
\begin{proof}
  By the implicit function theorem, the set $V = \{\bar{x} : f(\bx) = \bar{0}\}$ has local
  dimension $n-k$ at $\ba$.  Take $B$ a $\Qq$-definable box around
  $\ba$ such that $\dim(B \cap V) = n - k$.  Then $B \cap V$ is an
  $\Qq$-definable set of dimension $n-k$ containing $\ba$.  Since
  o-minimal dimension agrees with topological dimension, and o-minimal
  rank agrees with transcendence degree, it follows that
  $\trdeg(\ba/\Qq) \le n - k$.
\end{proof}
Say that a function $f : \Rr^n \to \Rr^m$ is \emph{additively
matrix-like} if it's induced by a matrix with entries in $\Zz$, i.e.,
\begin{equation*}
  f(x_1,\ldots,x_n) = \left( \sum_{j=1}^n \mu_{ij} x_j : 1 \le i \le m\right)
\end{equation*}
for some $m \times n$ matrix $\{\mu_{ij}\}_{1 \le i \le m, ~ 1 \le j
  \le n}$.  Say that a function $f : \Rr_{>0}^n \to \Rr_{>0}^m$ is
\emph{multiplicatively matrix-like} if it corresponds to an additively
matrix-like function under the logarithm map, or equivalently, it has
the form
\begin{equation*}
  f(x_1,\ldots,x_n) = \left( \prod_{j=1}^n x_j^{\mu_{ij}} : 1 \le i \le m\right)
\end{equation*}
for some $m \times n$ matrix $\{\mu_{ij}\}_{1 \le i \le m, ~ 1 \le j
  \le n}$.
\begin{lemma} \label{slip-sub}
  Let $\bb = (b_1,\ldots,b_n)$ be a tuple of positive elements in
  $\Rr$.  Then there is a tuple $\bc = (c_1,\ldots,c_k)$ of
  positive elements such that
  \begin{enumerate}
  \item The tuple $\bc$ is multiplicatively $\Qq$-linearly
    independent.
  \item $\bb = f(\bc)$ for some multiplicatively matrix-like function
    $f$.
  \item $\bc = g(\bb)$ for some multiplicatively matrix-like function
    $f$.
  \end{enumerate}
\end{lemma}
\begin{proof}
  Applying logarithms, we reduce to proving the analogous additive
  statement.  Suppose $b_1,\ldots,b_n \in \Rr$ are (additively)
  $\Qq$-linearly independent.  Then the group ($\Zz$-module) they
  generate is finitely generated and torsion-free (since $\Rr$ is).
  By the classification of finitely generated abelian groups, the
  generated group is free.  Take $c_1,\ldots,c_k$ a basis for this
  free group.  Then $\{b_1,\ldots,b_n\}$ and $\{c_1,\ldots,c_k\}$
  generate the same $\Zz$-module, so each tuple is a $\Zz$-linear
  combination of the other, showing that $\bb = f(\bc)$ and $\bc =
  g(\bb)$ for some additively matrix-like functions $f$ and $g$.
\end{proof}
\begin{lemma} \label{dumb}
  Suppose $\bc \in \Rr_{>0}^n$ is multiplicatively $\Qq$-linearly
  independent.  Suppose $f : \Rr_{>0}^n \to \Rr_{>0}^n$ is
  multiplicatively matrix-like and $f(\bc) = \bc$.  Then $f$ is the
  identity map.
\end{lemma}
\begin{proof}
  Using logarithms, it suffices to prove the equivalent additive
  statement.  Suppose $\bc \in \Rr^n$ is (additively) $\Qq$-linearly
  independent, and some integer matrix $M$ satisfies $M \cdot \bc = \bc$.  We
  must show that $M$ is the identity matrix $I_n$.  Let $M' = M -
  I_n$.  Then $M' \cdot \bc = 0$, and we must show that $M'$ vanishes.
  If $m_1,\ldots,m_n$ are the integers appearing in one of the rows of
  $M'$, then $\sum_{i=1}^n m_i c_i = 0$, implying that all the $m_i$
  vanish, as the $c_i$ are $\Qq$-linearly independent.
\end{proof}
Recall that $R_n$ is the ring of rational Laurent polynomials in
$x_1,\ldots,x_n,x_1^\alpha,\ldots,x_n^\alpha$.  The elements of $R_n$
are regarded as smooth functions $(\Rr_{>0})^n \to \Rr$.
\begin{remark} \label{spin-laurent}
  If $g : \Rr_{>0}^n \to \Rr_{>0}^m$ is multiplicatively matrix-like
  and $f \in R_m$, then $f \circ g \in R_n$.  This holds because for
  any $d_1,\ldots,d_n \in \Zz$, the two values
  \begin{equation*}
    x_1^{d_1}x_2^{d_2} \cdots x_n^{d_n} \text{ and } (x_1^{d_1} \cdots
    x_n^{d_n})^{\alpha}
  \end{equation*}
  are both Laurent polynomials in
  $x_1,\ldots,x_n,x_1^{\alpha},\ldots,x_n^{\alpha}$.
\end{remark}
\begin{proposition} \label{lockdown}
  If $f_1,\ldots,f_n \in R_n$ and $b_1,\ldots,b_n \in \Rr_{>0}$ are
  such that $(b_1,\ldots,b_n) \in V^{reg}(f_1,\ldots,f_n)$, then
  $\alpha$ is not in $\Qq(b_1,\ldots,b_n)^{\alg}$.
\end{proposition}
\begin{proof}
  First suppose that $b_1,\ldots,b_n$ are multiplicatively
  independent.  Let $u : \Rr_{>0}^n \to \Rr_{>0}^{2n}$ be the map
  $u(\bx) = (\bx,\bx^\alpha)$.  Write $f_i$ as $p_i(u(\bx))$ with
  $p_i$ a rational Laurent polynomial in $2n$ variables.  The map
  $(f_1,\ldots,f_n)$ is the the composition $(p_1,\ldots,p_n) \circ
  u$.  Taking derivatives and applying the chain rule, we see that the
  Jacobian matrices satisfy the identity
  \begin{equation*}
    \underbrace{J(f_1,\ldots,f_n)}_{n \times n} =
    \underbrace{J(p_1,\ldots,p_n)}_{n \times 2n} \cdot
    \underbrace{J(u)}_{2n \times n}.
  \end{equation*}
  At the point $\bb$, $J(f_1,\ldots,f_n)$ is invertible, and so
  $J(p_1,\ldots,p_n)$ has rank $n$ at the point $(\bb,\bb^\alpha)$.
  By Remark~\ref{sard-idk}, $\trdeg(\bb,\bb^\alpha/\Qq) \le 2n - n =
  n$.  If $\alpha \in \Qq(b_1,\ldots,b_n)^{\alg}$, then
  \begin{equation*}
    \trdeg(\bb,\bb^\alpha/\alpha) = \trdeg(\bb,\bb^\alpha/\Qq) -
    \trdeg(\alpha/\Qq) \le n - \trdeg(\alpha/\Qq) = n - 1,
  \end{equation*}
  contradicting Fact~\ref{f1}.

  Next consider the general case, where $\bb$ is not necessaringly
  multiplicatively independent.  Apply Lemma~\ref{slip-sub} to get a
  tuple $c_1,\ldots,c_k$ of positive elements in $\Rr$ such that
  \begin{gather*}
    \bb = g(\bc) \\
    \bc = h(\bb)
  \end{gather*}
  for some multiplicatively matrix-like functions $g$ and $h$, and
  such that $\bc$ is multiplicatively independent.  The functions $g,
  h$ are polynomials with coefficients in $\Zz$, so $\bb$ and $\bc$
  are interalgebraic over $\Qq$.  In particular, $\alpha$ is algebraic
  over $\bb$ iff it's algebraic over $\bc$.

  Note that $(h \circ g)(\bc) = \bc$.  By Lemma~\ref{dumb}, $h \circ
  g$ is the identity map.  By the chain rule, it follows that the
  Jacobian matrix $J(g)$ of $g$ at $\bc$ is injective (i.e., rank
  $k$).  Let $f^*_i = f_i \circ g$; then $f^*_i \in R_k$ by
  Remark~\ref{spin-laurent}.  By the chain rule,
  \begin{equation*}
    J(f^*_1,\ldots,f^*_n) = J(f_1,\ldots,f_n) \cdot J(g),
  \end{equation*}
  so $J(f^*_1,\ldots,f^*_n)$ is injective (i.e., rank $k$) at $\bc$.
  Permuting the $f^*_i$, we can assume that $J(f^*_1,\ldots,f^*_k)$
  has rank $k$.  Then
  \begin{itemize}
  \item $\bc$ is a tuple of positive real numbers which is
    multiplicatively independent.
  \item $\bc \in V^{reg}(f^*_1,\ldots,f^*_k)$, and each $f^*_i \in
    R_k$.
  \item $\Qq(\bc)^\alg = \Qq(\bb)^\alg$.
  \end{itemize}
  Replacing $\bb$ and $f_1,\ldots,f_n$ with $\bc$ and
  $f^*_1,\ldots,f^*_k$, we reduce to the multiplicatively independent
  case.
\end{proof}
Going forward, Proposition~\ref{lockdown} is the only thing we'll need from this section.

\section{The twisted derivative $\delta$} \label{delta-core}
Continue to assume that $\alpha \in \Rr_{>0}$ is exponentially
transcendental.  Let $K$ be the prime model of $\Rr_\alpha$, i.e., the
set $K = \dcl(\varnothing) \prec \Rr_\alpha$.  Note that if $f \in
R_n[\alpha]$ and $b_1,\ldots,b_n \in K_{>0}$, then $f(b_1,\ldots,b_n) \in
K$.
\begin{definition}
  An \emph{anchor} is a tuple $(b_1,\ldots,b_n,f_1,\ldots,f_n)$ with
  $b_1,\ldots,b_n \in \Rr_{>0}$ and $f_1,\ldots,f_n \in R_{n+1}$ such
  that
  \begin{equation*}
    (b_1,\ldots,b_n) \in V^{reg}(f_1(\alpha,-),\ldots,f_n(\alpha,-)).
  \end{equation*}
  A positive number $b$ is \emph{anchored by} an anchor $A =
  (b_1,\ldots,b_n,f_1,\ldots,f_n)$ if $b$ is one of the $b_i$.
\end{definition}
\begin{remark} \label{all-anchored}
  If $b \in K_{>0}$, then $b$ is 0-definable in $\Rr_\alpha$, so
  Corollary~\ref{cor1} gives $b_2,\ldots,b_n \in \Rr_{>0}$ and
  $g_1,\ldots,g_n \in R_n[\alpha]$ such that $(b,b_2,\ldots,b_n) \in
  V^{reg}(g_1,\ldots,g_n)$.  Each $g_i$ can be written as
  $f_i(\alpha,-)$ for some $f_i \in R_{n+1}$.  Then $b$ is anchored by
  $(b,b_2,\ldots,b_n,f_1,\ldots,f_n)$.

  Conversely, if $(b_1,\ldots,b_n,f_1,\ldots,f_n)$ is an anchor, then
  $(b_1,\ldots,b_n)$ is a member of the finite 0-definable set
  $V^{reg}(f_1(\alpha,-),\ldots,f_n(\alpha,-))$, so each $b_i$ is in
  $\dcl(\varnothing) = K$.  So the elements of $K_{>0}$ are exactly
  the elements of $\Rr_{>0}$ which can be anchored.
\end{remark}
\begin{lemma} \label{tfae-confusion}
  Suppose $f_1,\ldots,f_n \in R_{n+1}$ and $(\alpha,b_1,\ldots,b_n)
  \in V(f_1,\ldots,f_n)$, or equivalently, $(b_1,\ldots,b_n) \in
  V(f_1(\alpha,-),\ldots,f_n(\alpha,-))$.  The following are
  equivalent:
  \begin{enumerate}
  \item $(b_1,\ldots,b_n) \in
    V^{reg}(f_1(\alpha,-),\ldots,f_n(\alpha,-))$, i.e.,
    $(b_1,\ldots,b_n,f_1,\ldots,f_n)$ is an anchor.
  \item There is no non-zero vector $(d_1,\ldots,d_n) \in \Rr^n$ such
    that
    \begin{equation*}
      J(f_1(\alpha,-),\ldots,f_n(\alpha,-)) \cdot
      (d_1,\ldots,d_n)^{\mathrm{T}} = 0.
    \end{equation*}
    at $\bb$.
  \item There is a unique vector $(d_1,\ldots,d_n) \in \Rr^n$ such
    that
    \begin{equation*}
      \frac{\partial f_i}{\partial w}(\alpha,\bb) + \sum_{j=1}^n \frac{\partial
        f_i}{\partial x_j}(\alpha,\bb) \cdot d_j = 0 \text{ for } 1 \le i \le n, \tag{$\ast$}
    \end{equation*}
  where we think of $f_i$ as a function $f_i(w,x_1,\ldots,x_n)$.
  \item There is a unique vector $(d_1,\ldots,d_n) \in \Rr^n$ such
    that
    \begin{equation*}
      J(f_1,\ldots,f_n) \cdot (1,d_1,\ldots,d_n)^{\mathrm{T}} = 0. \tag{$\dag$}
    \end{equation*}
    at $(\alpha,\bb)$.
  \end{enumerate}
\end{lemma}
\begin{proof}
  The equivalence of (1) and (2) holds because
  $J(f_1(\alpha,-),\ldots,f_n(\alpha,-))$ is a square ($n \times n$)
  matrix, so it's invertible iff it has non-trivial nullspace.
  Equations ($\ast$) and ($\dag$) mean literally the same thing, so
  (3) and (4) are equivalent.  Note that equation ($\ast$) has the form
  \[ \text{(something)} + J(f_1(\alpha,-),\ldots,f_n(\alpha,-)) \cdot (d_1,\ldots,d_n)^{\mathrm{T}} = 0.\]
  If (2) fails then the uniqueness could not possibly hold in (3),
  since we could shift $\bar{d}$ by the vector from (2).  So (3)
  implies (2).  Conversely, (1) implies (3) because (1) says that
  $J(f_1(\alpha,-),\ldots,f_n(\alpha,-))$ is invertible.
\end{proof}

\begin{definition}
  The \emph{derivative} of an anchor $(b_1,\ldots,b_n,f_1,\ldots,f_n)$
  is the unique sequence $(d_1,\ldots,d_n) \in \Rr^n$ such that
  \begin{equation*}
    J(f_1,\ldots,f_n) \cdot (1,d_1,\ldots,d_n)^{\mathrm{T}} = 0,
  \end{equation*}
  at $\bb$.
\end{definition}
\begin{remark}
  The derivative $(d_1,\ldots,d_n)$ is a tuple in $K$, because it's
  the solution to a linear system of equations whose coefficients are
  all in $K$.  The coefficients are in $K$ because they come from
  partial derivatives $\frac{\partial f_i}{\partial w}$ and
  $\frac{\partial f_i}{\partial x_j}$ evaluated at a tuple in $K$.
  The ring $R_{n+1}[\alpha]$ is closed under partial derivation, and
  if you evaluate a function from $R_{n+1}[\alpha]$ at a tuple in $K$,
  the output value is in $K$.
\end{remark}
\begin{lemma} \label{1-to-0}
  If $(b_1,\ldots,b_n,f_1,\ldots,f_n)$ is an anchor with derivative
  $(d_1,\ldots,d_n)$, and $b_i = 1$, then $d_i = 0$.
\end{lemma}
\begin{proof}
  Permuting coordinates, we may assume $i=1$.  Suppose $b_1 = 1$ and
  $d_1 \ne 0$.  Let $g \in R_{n+1}$ be $g(w,x_1,\ldots,x_n) = x_1 -
  1$.  Then $(\alpha,b_1,\ldots,b_n) \in V(f_1,\ldots,f_n,g)$.  If
  $(\alpha,b_1,\ldots,b_n) \in V^{reg}(f_1,\ldots,f_n,g)$, then we
  contradict Proposition~\ref{lockdown}, since $\alpha \in
  \Qq(\alpha,b_1,\ldots,b_n)^{\alg}$.

  Instead, $(\alpha,\bb) \notin V^{reg}(f_1,\ldots,f_n,g)$.  Then the
  square matrix $J(f_1,\ldots,f_n,g)$ is non-invertible at
  $(\alpha,b_1,\ldots,b_n)$, so it has non-trivial kernel/nullspace.
  That is, there is a non-zero vector $(u_0,\ldots,u_n) \in \Rr^{n+1}$
  such that
  \begin{equation*}
    J(f_1,\ldots,f_n,g) \cdot (u_0,\ldots,u_n)^{\mathrm{T}} = 0.
  \end{equation*}
  This means that
  \begin{align*}
    J(f_1,\ldots,f_n) \cdot (u_0,\ldots,u_n)^{\mathrm{T}} &= 0 \tag{$\ast$} \\
    \nabla g \cdot (u_0,\ldots,u_n)^{\mathrm{T}} &= 0 \tag{$\dag$}
  \end{align*}
  The vector $\nabla g$ is just $(0,1,0,\ldots,0)$, so ($\dag$) means
  that $u_1 = 0$.  By ($\ast$), we must have $(u_0,\ldots,u_n) = c
  \cdot (1,d_1,\ldots,d_n)$ for some $c \in \Rr$.  So $0 = u_1 = c
  \cdot d_1$.  As $d_1 \ne 0$, we must have $c = 0$, and then
  $(u_0,\ldots,u_n)$ is the zero vector, a contradiction.
\end{proof}
\begin{remark} \label{concatenanchor}
  Suppose $A = (b_1,\ldots,b_n,f_1,\ldots,f_n)$ and $A^* =
  (b^*_1,\ldots,b^*_m,f^*_1,\ldots,f^*_m)$ are anchors.  Then
  $(b_1,\ldots,b_n,b^*_1,\ldots,b^*_m)$ is a regular solution of the
  system of equations
  \begin{align*}
    f_i(\alpha,x_1,\ldots,x_n) &= 0  \text{ for } 1 \le i \le n \\
    f^*_i(\alpha,y_1,\ldots,y_m) &= 0 \text{ for } 1 \le i \le m,
  \end{align*}
  so we get an anchor $A \oplus A^* :=
  (b_1,\ldots,b_n,b^*_1,\ldots,b^*_m,g_1,\ldots,g_n,g^*_1,\ldots,g^*_m)$
  where $g_i, g^*_j \in R_{1+n+m}$ are the functions
  \begin{align*}
    g_i(w,x_1,\ldots,x_n,y_1,\ldots,y_m) &= f_i(w,x_1,\ldots,x_n) \\
    g_j^*(w,x_1,\ldots,x_n,y_1,\ldots,y_m) &= f^*_j(w,y_1,\ldots,y_m).
  \end{align*}
  If $(d_1,\ldots,d_n)$ is the derivative of $A$ and
  $(d^*_1,\ldots,d^*_m)$ is the derivative of $A^*$, then
  \begin{equation*}
    J(g_1,\ldots,g_n,g^*_1,\ldots,g^*_m) \cdot (1,d_1,\ldots,d_n,d^*_1,\ldots,d^*_m)^{\mathrm{T}} = 0,
  \end{equation*}
  because the left hand side is a vector whose top half is
  $J(f_1,\ldots,f_n) \cdot (1,d_1,\ldots,d_n)^{\mathrm{T}}$ and bottom
  half is $J(f^*_1,\ldots,f^*_n) \cdot
  (1,d^*_1,\ldots,d^*_m)^{\mathrm{T}}$.  So the derivative of $A
  \oplus A^*$ is $(d_1,\ldots,d_n,d^*_1,\ldots,d^*_m)$.

  To summarize, if $(b_1,\ldots,b_n,\ldots)$ is an anchor with
  derivative $(d_1,\ldots,d_n)$ and $(b_1^*,\ldots,b_m^*,\ldots)$ is
  an anchor with derivative $(d_1^*,\ldots,d_m^*)$, then we can form a
  concatenated anchor \[(b_1,\ldots,b_n,b^*_1,\ldots,b^*_m,\ldots)\]
  with derivative $(d_1,\ldots,d_n,d^*_1,\ldots,d^*_m)$.
\end{remark}
The next three lemmas are about adding dummy variables for $x_i/x_j$
and $x_i+x_j$ and $x_i^\alpha$.
\begin{lemma} \label{division}
  Let $(b_1,\ldots,b_n,f_1,\ldots,f_n)$ be an anchor with derivative
  $(d_1,\ldots,d_n)$.  Suppose $i \ne j$.  Then there is an anchor
  $(b_1,\ldots,b_n,b_i/b_j,g_1,\ldots,g_{n+1})$ with derivative
  $(d_1,\ldots,d_n, (b_jd_i - b_id_j)/b_j^2)$.
\end{lemma}
\begin{proof}
  Permuting coordinates, we may assume $i=1$ and $j=2$.
  Let $g_1,\ldots,g_{n+1} \in R_{1+n+1}$ be the functions
  \begin{align*}
    g_i(w,x_1,\ldots,x_n,y) &= f_i(w,x_1,\ldots,x_n) \text{ for } 1 \le i \le n \\
    g_{n+1}(w,x_1,\ldots,x_n,y) &= yx_2 - x_1.
  \end{align*}
  Then $(\alpha,b_1,\ldots,b_n,b_1/b_2) \in V(g_1,\ldots,g_{n+1})$.
  \begin{claim}
    If $e_1,\ldots,e_{n+1} \in \Rr$, then
    \begin{equation*}
      J(g_1,\ldots,g_{n+1}) \cdot (1,e_1,\ldots,e_{n+1})^{\mathrm{T}} = 0 \iff
      (e_1,\ldots,e_{n+1}) = (d_1,\ldots,d_n, (b_2d_1 - b_1d_2)/b_2^2)
    \end{equation*}
  \end{claim}
  \begin{claimproof}
    The derivative $\nabla g_{n+1}$ is $(0,-1,y,0,0,\ldots,0,x_2)$, so
    the equation
    \begin{equation*}
      J(g_1,\ldots,g_{n+1}) \cdot (1,e_1,\ldots,e_{n+1})^{\mathrm{T}} = 0
    \end{equation*}
    is equivalent to the system
    \begin{align*}
      J(f_1,\ldots,f_{n+1}) \cdot (1,e_1,\ldots,e_n)^{\mathrm{T}} &= 0 \\
      (0,-1,b_1/b_2,0,0,\ldots,0,b_2) \cdot (1,e_1,\ldots,e_{n+1})^{\mathrm{T}} &= 0,
    \end{align*}
    or equivalently,
    \begin{align*}
      (e_1,\ldots,e_n) &= (d_1,\ldots,d_n) \\
      -e_1 + (b_1/b_2)e_2 + b_2 e_{n+1} &= 0.
    \end{align*}
    The unique solution is $(e_1,\ldots,e_n,e_{n+1}) =
    (d_1,\ldots,d_n,d_1/b_2 - b_1d_2/b_2^2)$.
  \end{claimproof}
  The claim then implies (see Lemma~\ref{tfae-confusion}) that
  $(b_1,\ldots,b_n,b_1/b_2,g_1,\ldots,g_{n+1})$ is an anchor with
  derivative $(d_1,\ldots,d_n,(b_2d_1 - b_1d_2)/b_2^2)$.
\end{proof}
\begin{lemma} \label{add}
  Let $(b_1,\ldots,b_n,f_1,\ldots,f_n)$ be an anchor with derivative
  $(d_1,\ldots,d_n)$.  Suppose $i \ne j$.  Then there is an anchor
  $(b_1,\ldots,b_n,b_i+b_j,g_1,\ldots,g_{n+1})$ with derivative
  $(d_1,\ldots,d_n, d_i+d_j)$.
\end{lemma}
\begin{proof}
  Similar to the previous lemma, adding a dummy variable $y$ and an
  equation $x_i + x_j - y = 0$.
\end{proof}
\begin{lemma} \label{power}
  Let $(b_1,\ldots,b_n,f_1,\ldots,f_n)$ be an anchor with derivative
  $(d_1,\ldots,d_n)$, and $i \le n$.  Then there is an anchor
  $(b_1,\ldots,b_n,b_i^\alpha,g_1,\ldots,g_{n+1})$ with derivative
  $(d_1,\ldots,d_n,\alpha b_i^{\alpha - 1} d_i)$.
\end{lemma}
\begin{proof}
  Similar to the previous two lemmas, adding a dummy variable $y$ and
  an equation $x_i^\alpha - y = 0$.  The new equation introduces the
  constraint
  \begin{equation*}
    \alpha b_i^{\alpha - 1} d_i - d_{n+1} = 0,
  \end{equation*}
  which ensures $d_{n+1} = \alpha b_i^{\alpha - 1} d_i$.
\end{proof}
\begin{lemma} \label{inner-concur}
  If $(b_1,\ldots,b_n,f_1,\ldots,f_n)$ is an anchor with derivative
  $(d_1,\ldots,d_n)$, and $b_i = b_j$, then $d_i = d_j$.
\end{lemma}
\begin{proof}
  We may assume $i \ne j$.  Then Lemma~\ref{division} gives an anchor
  $(b_1,\ldots,b_n,b_i/b_j,\ldots)$ with derivative
  $(d_1,\ldots,d_n,(b_jd_i - b_id_j)/b_j^2)$.  Since $b_i/b_j = 1$,
  Lemma~\ref{1-to-0} shows that $(b_jd_i - b_id_j)/b_j^2 = 0$, so
  $d_j/b_j = d_i/b_i$.  Since $b_j = b_i$, this implies $d_j = d_i$.
\end{proof}
\begin{lemma} \label{outer-concur}
  Let $A = (b_1,\ldots,b_n,f_1,\ldots,f_n)$ and $A^* =
  (b^*_1,\ldots,b^*_m,f^*_1,\ldots,f^*_m)$ be two anchors with
  derivatives $(d_1,\ldots,d_n)$ and $(d^*_1,\ldots,d^*_n)$,
  respectively.  If $b_i = b^*_j$, then $d_i = d^*_j$.
\end{lemma}
\begin{proof}
  By Remark~\ref{concatenanchor}, there is an anchor $A \oplus A^* =
  (b_1,\ldots,b_n,b^*_1,\ldots,b^*_m,\ldots)$ with derivative
  $(d_1,\ldots,d_n,d^*_1,\ldots,d^*_m)$.  By the previous lemma, if
  $b_i = b^*_j$ then $d_i = d^*_j$.
\end{proof}
\begin{definition}
  If $a \in K_{>0}$, then $\delta a$ is the element of $K$ such that
  for any anchor $(b_1,\ldots,b_n,f_1,\ldots,f_n)$ with derivative
  $(d_1,\ldots,d_n)$,
  \begin{equation*}
    a = b_i \implies \delta a = d_i.
  \end{equation*}
\end{definition}
This is well-defined by Remark~\ref{all-anchored} and
Lemma~\ref{outer-concur}.  Remark~\ref{all-anchored} shows every
positive $a \in K$ has at least one anchor, and
Lemma~\ref{outer-concur} shows that any two anchors calculate the same
derivative for $a$.

So if $A = (b_1,\ldots,b_n,f_1,\ldots,f_n)$ is an anchor, then the
derivative of $A$ is necessarily $(\delta b_1, \ldots, \delta b_n)$.
\begin{example} \label{d-alpha}
  Let $f(w,x) = w-x$.  Then $f(\alpha,x) = \alpha - x$, and $\alpha
  \in V^{reg}(f(\alpha,-))$, so $(\alpha,f)$ is an anchor.  The
  Jacobian matrix $Jf$ is the $1 \times 2$ matrix $\begin{pmatrix} 1 &
    -1 \end{pmatrix}$, whose nullspace is spanned by $(1,1)^{\mathrm{T}}$.  By Lemma~\ref{tfae-confusion}, the
  derivative of the anchor $(\alpha,f)$ is $1$, and therefore \[\delta \alpha =
  1.\]
\end{example}
\begin{lemma} \label{basic}
  Let $b_1, b_2$ be two positive elements of $K$.
  \begin{enumerate}
  \item $\delta(b_1/b_2) = (b_2 \delta b_1 - b_1 \delta b_2)/b_2^2$.
  \item $\delta(b_1 + b_2) = \delta b_1 + \delta b_2$.
  \item $\delta(b_1^\alpha) = \alpha b_1^{\alpha - 1} \delta b_1$.
  \end{enumerate}
\end{lemma}
\begin{proof}
  Let $A$ and $A^*$ anchor $b_1$ and $b_2$.  Moving to $A \oplus A^*$
  (Remark~\ref{concatenanchor}) we can assume $A = A^*$.  Permuting
  coordinates, we can assume $A =
  (b_1,b_2,b_3,\ldots,b_n,f_1,\ldots,f_n)$.  The derivative is
  $(\delta b_1, \ldots, \delta b_n)$.  By Lemma~\ref{division}, there
  is an anchor $(b_1,\ldots,b_n,b_1/b_2,\ldots)$ with derivative
  $(\delta b_1, \ldots, \delta b_n, (b_2 \delta b_1 - b_1 \delta
  b_2)/b_2^2)$.  Then $\delta (b_1 / b_2) = (b_2 \delta b_1 - b_1
  \delta b_2)/b_2^2$.  This proves (1).  The proofs of (2) and (3) are
  similar, using Lemmas~\ref{add} and \ref{power}.
\end{proof}
\begin{corollary} \label{div-to-mult}
  If $b_1, b_2 \in K_{>0}$, then $\delta (b_1 b_2) = b_1 \delta b_2 +
  b_2 \delta b_1$.
\end{corollary}
\begin{proof}
  Let $f : K_{>0} \to K$ be the map $b \mapsto \delta b / b$.
  Lemma~\ref{basic}(1) says that
  \begin{equation*}
    f(b_1/b_2) = f(b_1) - f(b_2),
  \end{equation*}
  so $f$ is a homomorphism $(K_{>0},\cdot) \to (K,+)$.  Then
  $f(b_1b_2) = f(b_1) + f(b_2)$, which means $\delta(b_1b_2) = b_1
  \delta b_2 + b_2 \delta b_1$.
\end{proof}
By Lemma~\ref{basic}(2), $\delta$ extends to an additive homomorphism
\begin{equation*}
  \delta : (K,+) \to (K,+),
\end{equation*}
characterized by the fact that $\delta(b-c) = \delta(b) - \delta(c)$
for $b,c \in K_{>0}$.  Then the identity
\begin{gather*}
  \delta(xy) = x \delta y + y \delta x
\end{gather*}
continues to hold, because both sides are bilinear in $x$ and $y$.
Reversing the proof of Corollary~\ref{div-to-mult}, one sees that
$\delta(x/y) = (y \delta x - x \delta y)/y^2$ for $x,y \in K$ with $y
\ne 0$.
\begin{theorem} \label{the-key}
  There is a unique derivation $\delta : K \to K$ satisfying the
  constraints
  \begin{align*}
    \delta \alpha &= 1 \tag{$\ast$} \\
    \delta x^\alpha &= \alpha x^{\alpha - 1} \delta x \text{ for } x > 0. \tag{$\dag$}
  \end{align*}
\end{theorem}
\begin{proof}
  We have just constructed such a derivation $\delta$.  Indeed,
  $\delta \alpha = 1$ by Example~\ref{d-alpha}, and $\delta x^\alpha =
  \alpha x^{\alpha - 1} \delta x$ by Lemma~\ref{basic}(3).

  Conversely, suppose $\delta'$ is another derivation satisfying the
  two constraints ($\ast$) and ($\dag$).  First note that if
  $f(x_1,\ldots,x_n) \in R_n$, then
  \begin{equation*}
    \delta' (f(b_1,\ldots,b_n)) = \sum_{i=1}^n \frac{\partial
      f}{\partial x_i}(b_1,\ldots,b_n) \cdot \delta' b_i, \tag{$\ddag$}
  \end{equation*}
  by induction on the complexity of $f$ using property ($\dag$).

  Given $b \in K$, we will show that $\delta' b = \delta b$.  We may
  assume $b > 0$.  Take an anchor $(b_1,\ldots,b_n,f_1,\ldots,f_n)$
  with $b = b_i$ (possible by Remark~\ref{all-anchored}).  Then
  \begin{equation*}
    (\alpha,b_1,\ldots,b_n) \in V^{reg}(f_1,\ldots,f_n).
  \end{equation*}
  Since $f_i(\alpha,b_1,\ldots,b_n) = 0$, equation ($\ddag$) shows that
  \begin{equation*}
    \frac{\partial f_i}{\partial w}(\alpha,b_1,\ldots,b_n) \cdot \delta'
    \alpha + \sum_{j=1}^n \frac{\partial f_i}{\partial x_j}(\alpha,b_1,\ldots,b_n) \cdot \delta' b_j \text{ for } 1 \le i \le n.
  \end{equation*}
  By property ($\ast$), $\delta' \alpha = 1$, and we can rewrite this
  as
  \begin{equation*}
    \frac{\partial f_i}{\partial w}(\alpha,b_1,\ldots,b_n) \cdot \delta'
    \alpha + \sum_{j=1}^n \frac{\partial f_i}{\partial x_j}(\alpha,b_1,\ldots,b_n) \cdot \delta' b_j \text{ for } 1 \le i \le n.    
  \end{equation*}
  This precisely means that $(\delta' b_1,\ldots, \delta' b_n)$ is the
  derivative of the anchor $(b_1,\ldots,b_n,f_1,\ldots,f_n)$.  This
  derivative is $(\delta b_1, \ldots, \delta b_n)$, so $\delta' b_i = \delta b_i$ and in particular $\delta' b =
  \delta b$.
\end{proof}
Henceforth, let $\delta$ be the unique derivation in
Theorem~\ref{the-key}.
\begin{remark} \label{clear-derivatives}
  As noted in the proof of Theorem~\ref{the-key}, for any
  $f(x_1,\ldots,x_n) \in R_n$ we have
  \begin{equation*}
    \delta f(a_1,\ldots,a_n) = \sum_{i=1}^n \frac{\partial f}{\partial
      x_i}(\ba) \cdot \delta a_i.
  \end{equation*}
  This fails for $R_n[\alpha]$, though: the constant function $f(x) =
  \alpha$ has
  \begin{equation*}
    \delta f(a) = \delta \alpha = 1 \ne 0 = f'(a) \cdot \delta a.
  \end{equation*}
\end{remark}
\begin{definition} \label{def-good}
  Let $D \subseteq \Rr^n$ and $f : D \to \Rr$ be $K$-definable
  (equivalently, 0-definable) in the structure $\Rr_\alpha$.  Then
  $(f,D)$ is \emph{good} if there is a $K$-definable function $g : D
  \times \Rr^n \to \Rr$ such that
  \begin{equation*}
    \delta f(b_1,\ldots,b_n) = g(b_1,\ldots,b_n,\delta b_1, \ldots,
    \delta b_n) \text{ for any } \bar{b} \in D(K).
  \end{equation*}
\end{definition}
\begin{remark} \label{good-cover}
  Suppose we have $D_1, D_2 \subseteq \Rr^n$ and $f : D_1 \cup D_2 \to
  \Rr$ are $K$-definable.  If $(f,D_i)$ is good for $i=1,2$, then
  $(f,D)$ is good.
\end{remark}
\begin{proof}
  There are $g_i : D_i \times \Rr^n \to \Rr$
  such that
  \begin{equation*}
    \delta f(\bb) = g_i(\bb,\delta \bb) \text{ if } \bb \in D_i(K).
  \end{equation*}
  Then we can take
  \begin{equation*}
    g(\bx,\by) = 
    \begin{cases}
      g_1(\bx,\by) & \text{ if } \bx \in D_1 \\
      g_2(\bx,\by) & \text{ if } \bx \in D_2 \setminus D_1
    \end{cases}
  \end{equation*}
  to arrange $\delta f(\bb) = g(\bb,\delta \bb)$.
\end{proof}
\begin{remark} \label{cell-project}
  Let $C \subseteq \Rr^n$ be a $k$-dimensional $K$-definable cell and
  $f : C \to \Rr$ be a $K$-definable function.  Let $\pi : \Rr^n \to
  \Rr^k$ be the coordinate projection such that $\pi(C)$ is an open
  cell in $\Rr^k$ and $C \to \pi(C)$ is a homeomorphism.  Let $f^* :
  \pi(C) \to \Rr$ be the $K$-definable function making the diagram
  commute:
  \begin{equation*}
    \xymatrix{C \ar[dr]^f \ar[d]_\pi & \\ \pi(C) \ar@{-->}[r]_{f^*} & \Rr}
  \end{equation*}
  If $(f^*,\pi(C))$ is good, then $(f,C)$ is good.
\end{remark}
\begin{proof}
  There is a $K$-definable function $g^*$ such that
  \begin{equation*}
    \delta f^*(\bx) = g^*(\bx, \delta \bx).
  \end{equation*}
  Let $g(\bx,\by) = g^*(\pi(\bx),\pi(\by))$.  Then
  \begin{equation*}
    \delta f(\bb) = \delta f^*(\pi(\bb)) = g^*(\pi(\bb),\delta \pi(\bb)) =
    g^*(\pi(\bb),\pi(\delta(\bb))) = g(\bb,\delta \bb). \qedhere
  \end{equation*}
\end{proof}
\begin{lemma} \label{implicit-case}
  Let $D \subseteq (\Rr_{>0})^n$ be open and $f : D \to \Rr_{>0}$ be a
  function, with both $f, D$ being $K$-definable.  If $f$ is
  implicitly defined in the sense of Definition~\ref{implicit}, then
  $(f,D)$ is good.
\end{lemma}
\begin{proof}
  By definition, there are functions $f_2,\ldots,f_m : D \to \Rr_{>0}$
  and functions $g^0_1,\ldots,g^0_m \in R_{n+m}[\alpha]$ such that
  \begin{equation*}
    (f(\ba),f_2(\ba),\ldots,f_m(\ba)) \in
    V^{reg}(g^0_1(\ba,-),\ldots,g^0_m(\ba,-))
  \end{equation*}
  for any $\ba \in U$.  Write $g^0_i(x_1,\ldots,x_n,y_1,\ldots,y_m)$
  as $g_i(\alpha,\bx,\by)$ for some $g_i(w,\bx,\by) \in R_{1+n+m}$.
  So then
  \begin{equation*}
    (f(\ba),\ldots,f_m(\ba)) \in V^{reg}(g_1(\alpha,\ba,-),\ldots,g_m(\alpha,\ba,-)) \text{ for all } \ba \in U.
  \end{equation*}
  For each $\ba \in U$ and $i \le m$, we have
  \begin{equation*}
    g_i(\alpha,\ba,f(\ba),f_2(\ba),\ldots,f_m(\ba)) = 0.
  \end{equation*}
  Applying $\delta$ to both sides and using
  Remark~\ref{clear-derivatives}, we see that
  \begin{equation*}
    \frac{\partial g_i}{\partial w}(\alpha,\ba,\bar{f}(\ba)) + \sum_{j=1}^n \frac{\partial g_i}{\partial x_j}(\alpha,\ba,\bar{f}(\ba)) \cdot \delta a_j + \sum_{j=1}^m \frac{\partial g_i}{\partial y_j}(\alpha,\ba,\bar{f}(\ba)) \cdot \delta (f_j(\ba)) \text{ for all } i \le m. \tag{$\ast$}
  \end{equation*}
  Since the matrix $\left(\frac{\partial g_i}{\partial
    y_j}\right)_{i,j \le m}$ is invertible at
  $(\alpha,\ba,\bar{f}(\ba))$, the system of equations ($\ast$) can be
  solved to calculate the value of $\delta(\bar{f}(\ba))$ as an
  algebraic function of the values
  \begin{gather*}
    \frac{\partial g_i}{\partial w}(\alpha,\ba,\bar{f}(\ba)) \text{ and } \frac{\partial g_i}{\partial x_j}(\alpha,\ba,\bar{f}(\ba)) \text{ and } \frac{\partial g_i}{\partial y_j}(\alpha,\ba,\bar{f}(\ba)) \text{ and } \delta \ba.
  \end{gather*}
  So then $\delta \bar{f}(\ba)$ is given by a $K$-definable function
  of $\ba$ and $\delta \ba$, and $f$ is good.
\end{proof}
\begin{theorem} \label{definability}
  Let $D \subseteq K^n$ and $f : D \to K$ be definable in the
  structure $K = (K,+,\cdot,\alpha,x^\alpha)$.  Then there is a
  definable function $g : D \times K^n \to K$ such that
  \begin{equation*}
    \delta f(b_1,\ldots,b_n) = g(b_1,\ldots,b_n, \delta b_1, \ldots,
    \delta b_n) \text{ for any } \bar{b} \in D.
  \end{equation*}
\end{theorem}
\begin{proof}
  We must show that $(f,D)$ is good in the sense of
  Definition~\ref{def-good}.  Proceed by induction on $\dim(D) + n$.
  Let
  \begin{gather*}
    D^+ = \{a \in D : f(a) > 0\} \\
    D^0 = \{a \in D : f(a) = 0\} \\
    D^- = \{a \in D : f(a) < 0\}.
  \end{gather*}
  By Remark~\ref{good-cover}, it suffices to deal with the restriction
  of $f$ to each of the sets $D^+$ and $D^0$ and $D^-$.  On $D^0$, $f$
  is the zero function, which is easily good (take $g = 0$).  So we
  may assume $f > 0$ or $f < 0$.  Replacing $f$ with $-f$ (which
  doesn't change goodness), we may assume $f > 0$ on $D$.

  For $a \in \Rr$, let $\sgn(a)$ be the element of $\{1,-1,0\}$ with
  the same sign as $a$.  Let $\approx$ be the equivalence relation on
  $\Rr^n$ given by
  \begin{equation*}
    \ba = \bb \iff (\sgn(a_1),\ldots,\sgn(a_n)) =
    (\sgn(b_1),\ldots,\sgn(b_n)),
  \end{equation*}
  with $3^n$ classes.  Partitioning $D$, we may assume that $D$ is a
  subset of one of the $\approx$-classes, meaning that
  $(\sgn(a_1),\ldots,\sgn(a_n))$ is constant as $\ba$ ranges over $D$.
  Changing coordinates, we may assume that $D \subseteq \Rr_{\ge
    0}^n$.

  Take a $K$-definable cell decomposition $D = \bigcup_{i=1}^n C_i$.
  By yet another application of Remark~\ref{good-cover}, it suffices
  to show that $f$ is good on each $C_i$, so we may assume that $D$ is
  a cell.  If $D$ is a non-open cell, then there is a coordinate
  projection $\pi : \Rr^n \to \Rr^k$ with $k < n$, such that $\pi(D)$
  is an open cell and $\pi : D \to \pi(D)$ is a homeomorphism.  Then
  $f : D \to \Rr$ factors as $D \stackrel{\pi}{\to} \pi(D)
  \stackrel{f^*}{\to} \Rr$, and $(f^*,\pi(D))$ is good by induction,
  so $f$ is good by Remark~\ref{cell-project}.

  Therefore, we may assume that $D$ is an open cell in $\Rr^n$.  Since
  $D \subseteq (\Rr_{\ge 0})^n$, $D$ must intersect $\Rr_{>0}^n$.  But
  $\Rr_{>0}^n$ is one of the $\approx$-classes, and $D$ lives in a
  single $\approx$-class, so $D \subseteq \Rr_{>0}$.  Then
  Corollary~\ref{cor2} gives 0-definable open sets $U_1,\ldots,U_m
  \subseteq D$ such that
  \begin{itemize}
  \item $f \restriction U_i$ is implicitly defined.
  \item The set $D_0 = D \setminus \bigcup_{i=1}^m U_i$ has $\dim(D_0)
    < n$.
  \end{itemize}
  By Lemma~\ref{implicit-case}, $f \restriction U_i$ is good, and by
  induction, $f \restriction D_0$ is good.  Since $U = D_0 \cup
  \bigcup_{i=1}^m U_i$, we see by Remark~\ref{good-cover} that $f$ is
  good.
\end{proof}

\section{The counterexample} \label{easy-part}
As in the previous two sections, let $\alpha \in \Rr_{>0}$ be exponentiall
transcendental, let $K$ be the prime model of $\Rr_\alpha$, and let
$\delta : K \to K$ be the derivation in Theorems~\ref{the-key} and
\ref{definability}.  Because $K = \dcl(\varnothing)$, ``definable''
and ``0-definable'' are equivalent in the structure $K$.

Let $M$ be the field of constants
\begin{equation*}
  M = \{x \in K : \delta x = 0\}.
\end{equation*}
\begin{remark} \label{6-1}
  If $x \in M_{>0}$, then $x^\alpha \in M$, because $\delta (x^\alpha)
  = \alpha x^{\alpha - 1} \delta x = 0$.  On the other hand, $\alpha
  \notin M$ because $\delta \alpha = 1 \ne 0$.
\end{remark}
\begin{remark} \label{dense-obvious}
  $M$ is dense in $K$, because $\Qq \subseteq M \subseteq K \subseteq
  \Rr$.
\end{remark}
\begin{lemma} \label{qe-for-induced}
  Let $D \subseteq K^{n+1}$ be 0-definable.  Let $\pi : K^{n+1} \to
  K^n$ be the coordinate projection onto the first $n$ coordinates.
  Then there is a 0-definable set $D' \subseteq K^n$ such that
  \begin{equation*}
    \pi(D \cap M^{n+1}) = D' \cap M^n.
  \end{equation*}
\end{lemma}
\begin{proof}
  We can write $D$ as a union of $K$-definable cells $C_i$, and
  thereby assume that $D$ is a cell.  There are two possibilities:
  \begin{enumerate}
  \item $D$ has the form $(f,g)_C = \{(x,y) \in K^n \times K : x \in
    C, f(x) < y < g(x)\}$ for some cell $C \subseteq K^n$ and
    continuous definable functions $f,g : C \to K \cup \{\pm \infty\}$
    such that $f<g$ on $C$.  In this case,
    \begin{align*}
      \pi(D \cap M^{n+1}) &= \{a \in C \cap M^n \mid \exists b \in K : f(a) < b < g(a)\} \\
      &= C \cap M^n,
    \end{align*}
    where the second equality holds because $M$ is dense in $K$
    (Remark~\ref{dense-obvious}.  So we can take $D' = C$.
  \item $D$ has the form $\Gamma(f)_C = \{(x,y) \in K^n \times K : x
    \in C, ~ y = f(x)\}$ for some cell $C \subseteq K^n$ and
    continuous definable function $f : C \to K$.  By
    Theorem~\ref{definability}, there is a definable function $g : C
    \times K^n \to K$ such that $\delta f(\ba) = g(\ba,\delta \ba)$
    for $\ba \in C$.  Then
    \begin{align*}
      \pi(D \cap M^{n+1}) &= \{a \in C \cap M^n \mid f(a) \in M \} \\
      &= \{a \in C \cap M^n \mid \delta f(a) = 0 \} \\
      &= \{a \in C \cap M^n \mid g(a,\delta a) = 0 \} \\
      &= \{a \in C \cap M^n \mid g(a,0) = 0\},
    \end{align*}
    because $a \in M^n \implies \delta a = 0$.  So we can take $D' =
    \{a \in C : g(a,0) = 0\}$.  \qedhere
  \end{enumerate}
\end{proof}
\begin{remark} \label{wom-theory}
  An $\Ll$-structure $\mathcal{M}$ has weakly o-minimal theory if and
  only if the following holds: for any $\Ll$-formula
  $\phi(x_1,\ldots,x_n,y)$, there is a bound $m_\phi$ depending on
  $\phi$ such that there do not exist
  $a_1,\ldots,a_n,b_1,\ldots,b_{m_\phi}$ with
  \begin{gather*}
    b_1 < \cdots < b_{m_\phi} \\
    \mathcal{M} \models \phi(a_1,\ldots,a_n,b_i) \iff i \equiv 0 \pmod{2}.
  \end{gather*}
\end{remark}
\begin{corollary} \label{induced-cor}
  Let $\mathcal{M}$ be $M$ with the 0-definable induced structure as a
  subset of $K$.
  \begin{enumerate}
  \item $\mathcal{M}$ has quantifier elimination, and in fact the
    0-definable sets in $\mathcal{M}$ all have the form $D \cap M^n$
    for some definable $D \subseteq K^n$.
  \item $\Th(\mathcal{M})$ is weakly o-minimal.
  \end{enumerate}
\end{corollary}
\begin{proof}
  \begin{enumerate}
  \item The family of sets of the form $D \cap M^n$ is obviously
    closed under boolean operations, and it's closed under projections
    by Lemma~\ref{qe-for-induced}.
  \item Let $\phi(x_1,\ldots,x_n,y)$ be a formula as in
    Remark~\ref{wom-theory}.  By the previous point,
    $\phi(\mathcal{M}) = D \cap M^{n+1}$ for some definable $D
    \subseteq K^{n+1}$.  By cell decomposition in $K$, we know there
    is a bound $m$ such that there do not exist
    $a_1,\ldots,a_n,b_1,\ldots,b_{m} \in K$ with
    \begin{gather*}
      b_1 < \cdots < b_{m} \\
      (a_1,\ldots,a_n,b_i) \in D \iff i \equiv 0 \pmod{2}.
    \end{gather*}
    Then the same holds in $M$ for the set $D \cap M^{n+1}$, and the
    criterion of Remark~\ref{wom-theory} is verified.  \qedhere
  \end{enumerate}
\end{proof}
\begin{theorem} \label{ending}
  Let $M = \{x \in K : \delta x = 0\}$.
  \begin{enumerate}
  \item If $x \in M$, then $x^\alpha \in M$.
  \item $\alpha \notin M$.
  \item The structure $(M,+,\cdot,x^\alpha)$ has a weakly o-minimal
    theory.
  \item In the structure $(M,+,\cdot,x^\alpha)$, the definable function
    \begin{align*}
      f : M_{>0} &\to M \\
      f(x) &= x^\alpha
    \end{align*}
    is nowhere differentiable.
  \end{enumerate}
\end{theorem}
\begin{proof}
	Points (1)--(2) were proved earlier (Remark~\ref{6-1}).
  \begin{enumerate}
  	\setcounter{enumi}{2}
  \item The function $x^\alpha$ is definable in the induced structure
    $\mathcal{M}$ appearing in Corollary~\ref{induced-cor}, so
    $(M,+,\cdot,x^\alpha)$ is a reduct of $\mathcal{M}$, and its
    theory is weakly o-minimal.
  \item If $x \in M_{>0}$, then $f'(x) = \alpha x^{\alpha - 1}$ isn't
    in $M$ because $x^{\alpha - 1} \in M$ and $\alpha \notin M$.  \qedhere
  \end{enumerate}
\end{proof}

\section{Directions for further research} \label{conclusion}
\subsection*{The model theory of $(K,\delta)$}
Of course, it would be nice to know something about the model theory
of the structure $(K,+,\cdot,x^\alpha,\delta)$.  More generally, can
we say anything about the model theory of models of
$\Th(K,+,\cdot,x^\alpha)$ expanded by ``twisted derivatives'' $\delta$
satisfying the conditions in Theorem~\ref{the-key}?
\subsection*{A C-minimal analogue}
I suspect that the weakly o-minimal field we have constructed has a
C-minimal analogue.  (See~\cite{c-source,cminfields} for background on C-minimality.)  Let $\Cc_p$ be the completion of $\Qq_p^{\alg}$.
For $\alpha \in \Cc_p$, the function
\begin{equation*}
  f_\alpha(x) = (1 + x)^\alpha = 1 + \alpha x + \frac{\alpha(\alpha-1)}{2}x^2 + \cdots
\end{equation*}
is defined on a sufficiently small ball around 0, and has the
derivative $\alpha f_\alpha(x)/x$.
\begin{conjecture}\label{c-nooo}
  There is some $\alpha \in \Cc_p$ and an algebraically closed
  subfield $M \subseteq \Cc_p$ closed under $f_\alpha$ such that
  \begin{itemize}
  \item The structure $(M,+,\cdot,v,f_\alpha)$ is C-minimal.
  \item $\alpha \notin M$, so $f_\alpha$ is nowhere differentiable
    within the structure $M$.
  \end{itemize}
\end{conjecture}
In fact, Conjecture~\ref{c-nooo} was the original goal, and the
present paper arose as a proof-of-concept for Conjecture~\ref{c-nooo}.
If I understand correctly, most of what we have done in this paper
should work equally well in the C-minimal case, with two big
exceptions:
\begin{itemize}
\item In Section~\ref{unneed}, we made extensive use of the fact that
  if $V$ is a smooth manifold and $p$ is a point not on $V$, then some
  sphere around $p$ is tangent to $V$.  (Namely: take a sphere around
  $p$ of minimal radius intersecting $V$.)  This probably doesn't hold
  in a non-archimedean setting.
\item Perhaps we don't have an analogue of the model completeness
  theorem of Miller~\cite{miller} (depending ultimately on the
  theorems of the complement by Gabrielov~\cite{gabrielov} and
  Wilkie~\cite{wilkie-thm}).
\end{itemize}
So the generalization to the C-minimal case is not entirely trivial.
\subsection*{Prospects for definable groups and fields}
In theories like o-minimal fields and P-minimal fields, the usual
strategy for understanding definable fields and groups involves
putting a definable $C^1$-manifold structure on any definable group.
This depends essentially on generic differentiability.  It appears to
be a major setback that generic differentiability fails in the weakly
o-minimal case, and possibly the C-minimal case.

Fortunately, there is a way to save everything.  Recall that a
sequence $\{a_i\}_{i \in \Nn}$ is a Cauchy sequence if $\lim_{n,m \to
  \infty} (a_n-a_m) = 0$.
\begin{definition}
  Let $K$ be a topological field.  A function $f : K \to K$ is
  \emph{Cauchy differentiable} at a point $a \in K$ if
  \begin{equation*}
    \lim_{x,y \to a} \left(\frac{f(x)-f(a)}{x-a} -
    \frac{f(y)-f(a)}{y-a}\right) = 0.
  \end{equation*}
\end{definition}
This ensures that $\lim_{x \to a} \frac{f(x)-f(a)}{x-a}$ exists in the
Cauchy completion of $K$.  One can also define Cauchy
differentiability for multivariable functions.

In forthcoming work, Acosta, Halevi, Hasson, and Peterzil prove the
following result, which nicely complements the negative result of this
paper:
\begin{theorem}[Acosta-Halevi-Hasson-Peterzil] \label{ahhp}
  Let $M$ be an expansion of an ordered field, with a weakly o-minimal
  theory.  If $U \subseteq M^n$ is an open definable set and $f : U
  \to M^m$ is a definable function, then $f$ is Cauchy differentiable
  on a dense open subset of $U$.
\end{theorem}
In particular, if $M$ is definably Cauchy complete (i.e., it defines
no non-valuational cuts), then $M$ has generic differentiability.

In \cite[Theorem~9.5]{own-C-min} I show that if $M$ is a definably
Cauchy complete C-minimal field of characteristic 0, then $M$ has
generic differentiability.  The same proof almost certainly proves
\begin{probably} \label{wellyeah}
  Let $M$ be a C-minimal expansion of a valued field of characteristic
  0.  If $U \subseteq M^n$ is an open definable set and $f : U \to
  M^m$ is a definable function, then $f$ is Cauchy differentiable on a
  dense open subset of $U$.
\end{probably}
It turns out that one can replace ``differentiable'' with ``Cauchy
differentiable'' in the $C^1$-manifolds machinery, and regain access
to tools like the adjoint representation.  If I understand correctly,
Acosta, Halevi, Hasson, and Peterzil carry this out in the weakly
o-minimal case in their forthcoming work, leading to results such as
the following:
\begin{theorem}[Acosta-Halevi-Hasson-Peterzil]
  Let $M$ be an expansion of an ordered field, with a weakly o-minimal
  theory.  If $K$ is a definable field in $M$, then $K$ is finite or
  definably isomorphic to $M$ or $M(\sqrt{-1})$.
\end{theorem}
I would assume that similar arguments work in the C-minimal case,
leading to results like the following:
\begin{probably}
  Let $M$ be a C-minimal expansion of a valued field of characteristic
  0.  If $K$ is a definable field in $M$, then $K$ is finite or
  definably isomorphic to $K$.
\end{probably}

\subsection*{Wencel completions}
The right perspective on all this may be to use \emph{Wencel
  completions}.  If $M$ is a weakly o-minimal field with no
  valuational cuts, Wencel~\cite{wencel} constructs a canonical
  ``o-minimal completion'' $\overline{M}$ of $M$, whose new points are
  the definable nonvaluational cuts in $M$.  For the structure
  $(M,+,\cdot,x^\alpha)$ we have constructed, the o-minimal completion
  is presumably just $(K,+,\cdot,x^\alpha)$.  My intuition from
  \cite{own-C-min} is that one should expect a similar Wencel-style
  completion for C-minimal fields.  I plan to investigate this in
  future work.

Now suppose that $K$ is either a nonvaluational weakly o-minimal field
or a characteristic 0 C-minimal field, and $f$ is a definable function
which is nowhere differentiable.  Let $\overline{K}$ be the Wencel
completion of $K$, which should be o-minimal or definably complete
C-minimal.  Then $\overline{K}$ has generic differentiability.  The
function $f$ should extend continuously to a definable function
$\overline{f}$ in the structure $\overline{K}$.  Then $\overline{f}$
is differentiable almost everywhere.  The derivative of $f$ really
does exist, but in $\overline{K}$ rather than $K$, explaining facts like Theorems~\ref{ahhp} and \ref{wellyeah}.

I suspect that the smoothest way to prove theorems about definable
fields and definable groups in these settings is to pass to the Wencel
completion $\overline{K}$, prove the theorems there (where generic
differentiability holds), and finally transfer the results back to
$K$.  In fact, in the weakly o-minimal case, this strategy has been
executed successfully by Eleftheriou~\cite{eleftheriou}.

\begin{remark}
  In some form, Wencel completions were the inspiration for
  Conjecture~\ref{c-nooo} and this entire paper.  \emph{If} one
  believes that Wencel completions work, then the failures of generic
  differentiability in the C-minimal case \emph{must} arise as
  follows:
  \begin{enumerate}
  \item Start with a C-minimal structure $\overline{M}$ with generic
    differentiability, and a definable function $\overline{f}$.
  \item Find a dense subfield $M$ such that
    \begin{itemize}
    \item $\overline{f}$ maps $M$ into $M$.
    \item The derivative $\overline{f}'$ maps $M$ into $\overline{M}
      \setminus M$.
    \item The structure $(M,+,\cdot,v,f)$ is C-minimal.
    \end{itemize}
  \end{enumerate}
  The simplest example of a function $\overline{f}$ for which this
  might work is $\overline{f}(x) = x^\alpha$, with $\alpha$
  transcendental.
\end{remark}

\appendix

\section{Appendix: Weakly o-minimal theories of fields have the exchange property} \label{app}

Recall that a complete theory $T$ has the \emph{exchange
property}---or more accurately, $\acl$ has the exchange property in
$T$---if in any model $M \models T$, one has
\begin{equation*}
  a \in \acl(Cb) \setminus \acl(C) \implies b \in \acl(Ca)
\end{equation*}
for $a,b \in M$ and $C \subseteq M$.  This condition means that $\acl$
is a pregeometry on $M$.  When $\acl$ has the exchange property, there
is a nice dimension theory on definable sets
\cite[Section~2]{udi-anand-group-field}, and other nice things happen
in topologically tame settings (see for example
\cite[Proposition~3.31, Corollary~3.35]{visceral}).  The exchange
property holds in o-minimal structures, as well as P-minimal
structures \cite[Theorem~6.2]{p-min} and C-minimal expansions of
valued fields \cite{own-C-min}.

As noted earlier, Macpherson, Marker, and Steinhorn ask whether weakly
o-minimal expansions of ordered fields have the exchange property
\cite[\S7.3, Problem~4]{weakOmin}.  We will answer this question
positively---at least in the case of weakly o-minimal \emph{theories}
of ordered fields, what one could call ``strongly weakly o-minimal
fields''.

The core idea of the proof was inspired by Acosta's proof of generic
local convexity of definable functions in weakly o-minimal fields:
\begin{fact}[Acosta] \label{acofact}
  Let $M$ be an expansion of an ordered field, with a weakly o-minimal
  theory.  Let $U \subseteq M$ be definable and open and let $f : U
  \to M$ be definable.  Then there is a cofinite subset $U_0 \subseteq
  U$ such that for any $a \in U_0$, the function $f$ is strictly
  convex, strictly concave, or linear on a neighborhood of $a$.
\end{fact}
(Fact~\ref{acofact} should appear implicitly or explicitly in the
forthcoming work of Acosta-Halevi-Hasson-Peterzil.)

The original strategy for proving the exchange property involved
proving a mysterious variant of Fact~\ref{acofact} for locally
constant functions.\footnote{Let $f$ be a function which is locally
constant and weakly increasing.  Say that $f$ is \emph{pseudoconvex}
if $f(x)<f(y)<f(z)$ implies $\frac{f(x)-f(y)}{x-y} <
\frac{f(y)-f(z)}{y-z}$, and define \emph{pseudoconcave} similarly.
The strategy was to prove that $f$ is locally pseudoconvex or
pseudoconcave.}  While we no longer use convexity here, you can see a
trace of this idea in Lemma~\ref{tricol}.

We aim to make the proof be self-contained.  Some steps could be
simplified or eliminated using tools from \cite{weakOmin}.  Peterzil
has found a way to considerably simplify the proof given here; his
proof should appear in the forthcoming work of
Acosta-Halevi-Hasson-Peterzil.

For the remainder of the appendix, assume that $T$ is a complete,
weakly o-minimal expansion of the theory of ordered fields.  Let $\Mm$
be a monster model of $T$, that is, a $\kappa$-saturated and
$\kappa$-strongly homogeneous model for $\kappa$ much larger than the
size of the language.

\subsection*{Tools}
\begin{remark} \label{ideal}
  Every infinite unary definable set $D \subseteq \Mm$ contains an
  infinite definable convex set, by weak o-minimality.  If $I$ is an
  infinite definable convex set, and we partition $I$ into finitely
  many definable sets $I = D_1 \sqcup \cdots \sqcup D_n$, then one of
  the $D_i$ is infinite and therefore contains an infinite definable
  convex set.  We will use these facts without comment, below.
\end{remark}
The following is well-known, but we include the easy proof for
completeness:
\begin{lemma} \label{hammer}
  Let $I \subseteq \Mm$ be an infinite definable convex set.  Let
  $\bowtie$ be a definable relation between $I$ and $\Mm_{>0}$
  satisfying the following conditions:
  \begin{itemize}
  \item For any $a \in I$, there is $\epsilon > 0$ such that $a
    \bowtie \epsilon$.
  \item If $a \bowtie \epsilon$ and $0 < \epsilon' < \epsilon$, then
    $a \bowtie \epsilon'$.
  \end{itemize}
  Then there is a smaller infinite definable convex set $I_0 \subseteq
  I$ and some positive $\epsilon_0$ such that for every $a \in I_0$,
  we have $a \bowtie \epsilon_0$.
\end{lemma}
\begin{proof}
  Take $a_1, a_2, \ldots$ distinct in $I$.  Take $\epsilon_i > 0$ such
  that $a_i \bowtie \epsilon_i$.  By saturation, there is some
  positive $\epsilon_0$ such that $\epsilon_0 \le \epsilon_i$ for
  every $i$.  Then $a_i \bowtie \epsilon_0$ for every $i$.  Let $D =
  \{a \in I : a \bowtie \epsilon_0\}$.  Then $D \supseteq \{a_i : i
  \in I\}$, so $D$ is infinite.  Take $I_0$ to be an infinite convex
  component of $D$.
\end{proof}
\begin{lemma} \label{2-indisc}
  Let $I \subseteq \Mm$ be an infinite definable convex set.  Let
  $\phi(x,y)$ be an $\Ll_{\Mm}$-formula in two variables.  Then there
  is an infinite definable convex subset $I_0 \subseteq I$ such that
  one of the following holds:
  \begin{itemize}
  \item For all $a,b \in I_0$ with $a<b$, we have $\phi(a,b)$.
  \item For all $a,b \in I_0$ with $a<b$, we have $\neg \phi(a,b)$.
  \end{itemize}
\end{lemma}
\begin{proof}
  Let $D_1$ (resp.\@ $D_0$) be the set of $a \in I$ such that there is
  $\epsilon > 0$ such that $\phi(a,x)$ is true (resp.\@ false) for $x
  \in (a,a+\epsilon)$.  By weak o-minimality, every $a \in I$ belongs
  to $D_0 \cup D_1$.  Shrinking $I$ we may assume that $I \subseteq
  D_1$ or $I \subseteq D_0$ (see Remark~\ref{ideal}).  Replacing
  $\phi$ with $\neg \phi$, we may assume that $I \subseteq D_1$.  Let
  $a \bowtie \epsilon$ mean that $\phi(a,x)$ holds for $x \in
  (a,a+\epsilon)$.  For every $a \in I$ there is $\epsilon > 0$ such
  that $a \bowtie \epsilon$.  By Lemma~\ref{hammer} we can find an
  infinite definable convex set $I_0 \subseteq I$ and an $\epsilon >
  0$ such that $a \bowtie \epsilon$ for every $a \in I_0$.  By
  replacing $I_0$ with a small interval contained inside $I_0$, we can
  assume the diameter of $I_0$ is less than $\epsilon$:
  \begin{equation*}
    a,b \in I_0 \implies |a-b| < \epsilon.
  \end{equation*}
  Now for any $a<b$ in $I_0$, we have $b \in (a,a+\epsilon)$, and $a
  \bowtie \epsilon$, so $\phi(a,b)$ holds.
\end{proof}
As an example, we can reprove the local monotonicity theorem for
definable functions $f : I \to \Mm$, a case of
\cite[Theorem~3.3]{weakOmin}:
\begin{example} \label{monot}
  Let $I$ be an infinite definable convex set.  Let $Y$ be a
  definable linearly ordered set (such as $\Mm$) and let $f : I
  \to Y$ be a definable function.  Then there is a smaller infinite
  definable convex set $I_0 \subseteq I$ such that one of the
  following holds for $x,y \in I_0$:
  \begin{gather*}
    x < y \implies f(x) < f(y).  \\
    x < y \implies f(x) \ge f(y).
  \end{gather*}
This follows by applying Lemma~\ref{2-indisc} to the formula $x<y$.
  In the second case, we can restrict further, and get one of the following two cases:
  \begin{gather*}
    x<y \implies f(x) = f(y). \\
    x<y \implies f(x) > f(y).
  \end{gather*}
  In particular, there is an infinite definable convex subset $I_0
  \subseteq I$ on which $f$ is strictly increasing, strictly
  decreasing, or constant.
\end{example}
Suppose $I$ is a definable convex set and $E$ is a definable
equivalence relation on $I$ that is \emph{convex}, i.e., every
equivalence class is convex.  Say that a subset $X \subseteq I$ is
\emph{$E$-infinite} if the image of $X$ in $I/E$ is infinite.
\begin{remark} \label{ideal-2}
  If $X \cup Y$ is $E$-infinite, then $X$ or $Y$ is $E$-infinite.  It
  follows that if $D \subseteq I$ is definable and $E$-infinite, then
  $D$ contains an $E$-infinite definable convex subset $I_0$.
\end{remark}
We need the following variant of Lemma~\ref{hammer}:
\begin{lemma} \label{hammer-2}
  Let $I \subseteq \Mm$ be a definable convex set and $E$ be a
  definable convex equivalence relation with infinitely many classes.
  Let $\bowtie$ be a definable relation between $I$ and $\Mm_{>0}$
  satisfying the following conditions:
  \begin{itemize}
  \item For any $a \in I$ there is $\epsilon > 0$ such that $a \bowtie
    \epsilon$.
  \item If $a \bowtie \epsilon$ and $0 < \epsilon' < \epsilon$, then
    $a \bowtie \epsilon'$.
  \end{itemize}
  Then there is an $E$-infinite definable convex subset $I_0 \subseteq
  I$ and an $\epsilon_0 > 0$ such that every $a \in I_0$ satisfies $a
  \bowtie \epsilon_0$.
\end{lemma}
\begin{proof}
  Take $a_1, a_2, \ldots \in I$ in distinct $E$-classes.  As in the
  proof of Lemma~\ref{hammer}, take $\epsilon_0 > 0$ such that $a_i
  \bowtie \epsilon_0$ for every $i$.  Then $D = \{a \in I : a \bowtie
  \epsilon_0\}$ is $E$-infinite; take $I_0$ an $E$-infinite convex
  subset of $D$.
\end{proof}
If $f : I \to (X,\le)$ is weakly increasing, then the equivalence
relation
\begin{equation*}
  E(x,y) \iff f(x)=f(y)
\end{equation*}
is convex, and then we define \emph{$f$-infinite} to mean
$E$-infinite.  More simply, $X \subseteq I$ is $f$-infinite if the
image $f(X)$ is infinite.  We will mostly apply Lemma~\ref{hammer-2}
in this case.

\subsection*{Proof of the exchange property}
\begin{theorem} \label{fin-image-1}
  Let $I \subseteq \Mm$ be a definable convex set and $f : I \to \Mm$
  be a definable function which is locally constant.  Then the image
  $f(I)$ is finite.
\end{theorem}
We give the proof as a series of lemmas.  Assume until
Theorem~\ref{exchange} that we have a locally constant definable
function $f : I \to \Mm$ with infinite image.
\begin{lemma} \label{convex-fibers}
  There is a definable convex subset $I_0 \subseteq I$ such that
  $f(I_0)$ is infinite, and the fibers of $f : I_0 \to \Mm$ are
  convex.
\end{lemma}
\begin{proof}
  Let $D$ be the definable set of $a \in I$ such that $a$ belongs to
  the leftmost (i.e., the most negative) convex component of
  $f^{-1}(f(a))$.  Let $g$ be the restriction $f \restriction D$.
  Then $f$ and $g$ have the same image, and for any $b$ in the image,
  $g^{-1}(b)$ is the leftmost convex component of $f^{-1}(b)$.  In
  particular, $g(D)$ is infinite and the fibers of $g$ are convex.
  Take $I_0$ to be a convex component of $D$ with infinite image under
  $g$.
\end{proof}
Replacing $I$ with $I_0$, we assume from now on that the fibers of $f$
are convex.  Let $E$ be the convex equivalence relation $E(x,y) \iff
f(x)=f(y)$.
\begin{lemma}
  There is an $E$-infinite definable convex subset $I_0 \subseteq I$
  such that $f$ is weakly increasing or decreasing on $I_0$.
\end{lemma}
\begin{proof}
  Since $E$ is convex, the quotient $I/E$ is a linearly ordered set.
  The map $f$ induces a bijection $I/E \to \im(f)$.  Let $g : \im(f)
  \to I/E$ be the inverse map, another definable bijection.  By
  Example~\ref{monot}, there is an infinite convex set $J \subseteq
  \im(f)$ on which $g$ is strictly increasing or strictly decreasing.
  Then $f$ is weakly increasing or decreasing on $D = \{x \in I : f(x)
  \in J\}$, and the image $f(D)=J$ is infinite, so $D$ is
  $E$-infinite.  Take $I_0$ an $E$-infinite convex component of $D$.
\end{proof}
Replacing $I$ with $I_0$, and $f(x)$ with $\pm f(x)$, we assume from
now on that $f$ is weakly increasing.
\begin{lemma} \label{silly-open}
  There is an $E$-infinite definable convex subset $I_0 \subseteq I$
  such that $f(I_0)$ is an open interval $(a,b)$ in $\Mm$.
\end{lemma}
\begin{proof}
  The image $f(I)$ is infinite, so it contains an open interval $J$.
  Take $I_0 = f^{-1}(J)$.
\end{proof}
Replacing $I$ with $I_0$, we assume from now on that the image $f(I)$
is an open interval in $\Mm$.

For $x,y \in I$, let $x \ll y$ mean that $f(x) < f(y)$, or
equivalently, $[x]_E < [y]_E$ where $[x]_E$ is the $E$-equivalence
class of $x$.  Let $\phi(x,y)$ be the $\Ll(\Mm)$-formula saying that
$x,y \in I$ and $x \ll y$ and the line segment with endpoints
$(x,f(x))$ and $(y,f(y))$ doesn't intersect the graph $\Gamma(f)$
except at its endpoints.
\begin{lemma}\label{secant}
  If $a \in I$ and $a \ll x$, then there is $a \ll x' \le x$ with
  $\phi(a,x')$.
\end{lemma}
\begin{proof}
  Let $\ell$ be the line segment from $(x,f(x))$ to $(a,f(a))$.  Since
  $x \ll a$ we have $f(x) < f(a)$, so $\ell$ has positive slope.
  Since $f$ is locally constant, $\ell$ cannot intersect $\Gamma(f)$
  on an open interval.  Therefore, it has only finitely many points of
  intersection.  Take $x'$ to be the leftmost intersection, other than $a$:
  \begin{equation*}
    x' = \min \{x' \in (a,x] : (x',f(x')) \in \ell\}.
  \end{equation*}
  Then $a < x' \le x$.  Since $\ell$ has positive slope, $f(a) <
  f(x')$, so $a \ll x' \le x$.  Clearly $\phi(a,x')$ holds.
\end{proof}
\begin{lemma}\label{tricol}
  There is an $E$-infinite definable convex subset $I_0 \subseteq I$
  such that the image $f(I_0)$ is an open interval, and no line of
  positive slope $\ell$ intersects the graph of $f \restriction I_0$
  at more than two points.
\end{lemma}
\begin{proof}
  For $a \in I$ and $\epsilon > 0$, let $a \bowtie \epsilon$ mean that
  \begin{equation*}
    f(a) < f(x) < f(a) + \epsilon \implies \phi(a,x).
  \end{equation*}
  \begin{claim}
    For every $a \in I$, there is an $\epsilon > 0$ with $a \bowtie
    \epsilon$.
  \end{claim}
  \begin{claimproof}
    Let $I_{\gg a} = \{x \in I : x \gg a\}$.  By Lemma~\ref{secant},
    the set $\phi(a,\Mm)$ is coinitial (downwards cofinal) in $I_{\gg
      a}$.  By weak o-minimality, all sufficiently small elements of
    $I_{\gg a}$ satisfy $\phi(a,x)$.  Therefore, there is an element
    $b \in I_{\gg a}$ such that
    \begin{equation*}
      a \ll x \le b \implies \phi(a,x).
    \end{equation*}
    Take $\epsilon < f(b) - f(a)$.  Then
    \begin{equation*}
      f(a) < f(x) < f(a) + \epsilon \implies f(a) < f(x) < f(b) \implies a \ll
      x \ll b \implies \phi(a,x).
    \end{equation*}
    Thus $a \bowtie \epsilon$ holds.
  \end{claimproof}
  By Lemma~\ref{hammer-2}, there is an $E$-infinite definable convex
  subset $I_1 \subseteq I$ and an $\epsilon>0$ such that $a \in I_1
  \implies a \bowtie \epsilon$.  The image $f(I_1)$ is infinite, so it
  contains an open interval of the form $J = (b,b+\delta)$ with
  $\delta < \epsilon$.  Let $I_0 = f^{-1}(J)$.  Then $f(I_0)$ is the
  open interval $J$.  For the other required property, suppose for the
  sake of contradiction that a line $\ell$ of positive slope hits the
  graph of $f \restriction I_0$ at three points $(x,f(x)), (y,f(y)),
  (z,f(z))$ with $x<y<z$.  Then $\neg \phi(x,z)$ holds by definition
  of $\phi$.  On the other hand,
  \begin{equation*}
    f(x) < f(z) < f(x) + \delta < f(x) + \epsilon, \text{ and } x
    \bowtie \epsilon,
  \end{equation*}
  so $\phi(x,z)$ holds, a contradiction.
\end{proof}
Replacing $I$ with $I_0$, we assume from now on that no line of
positive slope $\ell$ hits the graph of $f$ at three points.

For $x,y \in I$, let $m(x,y) = \frac{f(x)-f(y)}{x-y}$. (Thanks to Acosta, Halevi, Hasson, and Peterzil for suggesting this notation.)
\begin{lemma}
  If $a \in I$, then all sufficiently small positive $\epsilon$ are in
  the set $D = \{m(a,x) : x \gg a\}$.
\end{lemma}
\begin{proof}
  By weak o-minimality, it suffices to show that $D$ contains
  arbitrarily small elements.  Since $f$ is locally constant, there is
  $c > 0$ such that $f$ is constant on $[a,a+c]$.  Since $\im(f)$ is
  an open interval, there are $x \in I$ such that $f(x) - f(a)$ is an
  arbitrarily small positive number.  Such an $x$ must have $x \gg a$
  so $x > a+c$.  Then
  \begin{equation*}
    m(a,x) = \frac{f(x)-f(a)}{x-a} < \frac{f(x)-f(a)}{c}
  \end{equation*}
  is arbitrarily small.
\end{proof}
Similarly, the set $D' = \{m(a,y) : y \ll a\}$ must contain all
sufficiently small positive numbers.  Then some positive $\epsilon >
0$ is in both $D$ and $D'$, meaning that there are $x,y$ with
\begin{equation*}
  y \ll a \ll x \text{ and } m(a,x) = m(a,y).
\end{equation*}
Then the three points $(x,f(x))$ and $(a,f(a))$ and $(y,f(y))$ are
collinear, a contradiction.  This completes the proof of
Theorem~\ref{fin-image-1}.

Recall that a complete theory is \emph{geometric} in the sense of
Hrushovski and Pillay \cite[Definition~2.1]{udi-anand-group-field} if $T$ eliminates
$\exists^\infty$ and acl has the exchange property.
\begin{theorem} \label{exchange}
  $T$ is geometric.
\end{theorem}
\begin{proof}
  Since $T$ is weakly o-minimal and densely ordered, it clearly
  eliminates $\exists^\infty$: a definable set $D \subseteq \Mm$ is
  infinite iff it contains an open interval.  It remains to show that
  $\acl$ has the exchange property.  Suppose $a \in \acl(bC) \setminus
  \acl(C)$.  We will show that $b \in \acl(aC)$.  Since $\Mm$ is
  0-definably linearly ordered, $\acl$ and $\dcl$ agree, so $a \in
  \dcl(bC)$, meaning that $a = f(b)$ for some $C$-definable function
  $f$.  If $b$ is in the finite $C$-definable boundary of $f^{-1}(a)$,
  then $b \in \acl(aC)$ as desired.  Otherwise, $b$ is in the interior
  of $f^{-1}(a)$, meaning that $f$ is locally constant at $b$.  Let
  $U$ be the largest open subset of $\Mm$ on which $f$ is locally
  constant.  Then $U$ is $C$-definable and contains $b$.  By
  Theorem~\ref{fin-image-1} applied to the convex components of $U$,
  the image $f(U)$ is finite, so its member $a = f(b) \in f(U)$ is
  algebraic over $C$, a contradiction.
\end{proof}

\begin{acknowledgment}
  The author was supported by the National Natural Science Foundation
  of China (Grant No.\@ W2532009).  The author would like to thank the
  people who told him about the references \cite{miller,bkw,servi},
  including Jan Dobrowolski and many of the participants at the 2026
  Oberwolfach meeting ``Geometry and Combinatorics: the Model
  Theoretic Perspective'', where this work was presented.  Special
  thanks go to whoever pointed out \cite{bkw}; this significantly
  reduced the length of the paper.  Thanks also to Assaf Hasson for
  informing the author of his forthcoming work with Acosta, Halevi,
  and Peterzil, as well as the classic work of Wencel~\cite{wencel}.
  Lastly, the author would like to thank all four of Acosta,
  Halevi, Hasson, and Peterzil for providing the inspiration for the
  appendix and for reviewing earlier versions of the appendix.
\end{acknowledgment}

\bibliographystyle{plain} \bibliography{mybib}{}

\end{document}